\documentclass[12pt]{amsart}

\usepackage{amsfonts}
\usepackage{mathrsfs}
\usepackage{amsmath}
\usepackage{amssymb}
\usepackage{amsthm}
\usepackage{enumitem}
\usepackage{latexsym}
\usepackage{color}
\usepackage{geometry}
\theoremstyle{definition}
\newtheorem{defin}{Definition}[section]

\newtheorem{lem}[defin]{Lemma}

\newtheorem{teo}[defin]{Theorem}

\newtheorem*{lemnumber}{Lemma}

\DeclareMathOperator{\ran}{ran}
\newtheoremstyle{TheoremNum}
{\topsep}{\topsep}              %%% space between body and thm
{\itshape}                      %%% Thm body font
{}                              %%% Indent amount (empty = no indent)
{\bfseries}                     %%% Thm head font
{.}                             %%% Punctuation after thm head
{ }                             %%% Space after thm head
{\thmname{#1}\thmnote{ \bfseries #3}}%%% Thm head spec
\theoremstyle{TheoremNum}

\newtheorem{lemn}{Lemma}

\DeclareMathOperator{\supp}{ \text{supp }}

\newcommand{\void}{\varnothing}
\newcommand{\bs}{\backslash}

\newcommand{\comment}[1]{}
\begin{document}

	\title{On $p$-compact group topologies on direct sums of ${\mathbb Q}$}
	\author[M. K. Bellini]{Matheus Koveroff Bellini}
	\author[V. O. Rodrigues]{Vinicius de Oliveira Rodrigues}
	\author[A. H. Tomita]{Artur Hideyuki Tomita}
	\address{Depto de Matem\'atica, Instituto de Matem\'atica e Estat\'istica, Universidade de S\~ao Paulo, Rua do Mat\~ao, 1010 -- CEP 05508-090, S\~ao Paulo, SP - Brazil}
	\email{tomita@ime.usp.br, matheusb@ime.usp.br, vinior@ime.usp.br}
	\thanks{The first listed author has received financial support from FAPESP 2017/15709-6.}
	\thanks{The second listed author has received financial support from FAPESP 2017/15502-2.}
	\thanks{The third listed author has received financial support from FAPESP 2016/26216-8.}
	\date{}
	\commby{}

	\begin{abstract}
		We prove that if $p$ is a selective ultrafilter then ${\mathbb Q}^{(\kappa)}$ has a $p$-compact group topology without non-trivial convergent sequences, for each infinite cardinal $\kappa =\kappa^\omega$. In particular, this gives the first arbitrarily large examples of countably compact groups without non-trivial convergent sequences that are torsion-free.
	\end{abstract}
	
	\maketitle
	
	\section{Introduction}
	
	\subsection{Some history}
	
	Halmos proved that ${\mathbb R}$ can be endowed with a compact group topology, which in particular contains non-trivial convergent sequences \cite{halmos}. Tkachenko and Yashenko showed from Martin's Axiom that ${\mathbb R}$ can be endowed with a countably compact group topology without non-trivial convergent sequences \cite{tkachenko&yaschenko}. Madariaga-Garcia and Tomita assumed the existence of $2^{\mathfrak c}$ selective ultrafilters to obtain a countably compact group topology on the free Abelian group of cardinality $2^{\mathfrak c}$ without non-trivial convergent sequences. The construction does not yield any example of larger cardinality \cite{madariaga-garcia&tomita}.
	
	Fuchs showed that there are no compact free Abelian groups \cite{fuchs1970infinite} and Tomita showed that there are no free Abelian groups whose countable power is countably compact \cite{tomita2}. In particular, there are no $p$-compact free Abelian groups. Using what was called integer stacks, Tomita showed that there exists a group topology in the free Abelian group with ${\mathfrak c}$ generators whose every finite power is countably compact \cite{tomita2015}.
	
	Assuming the existence of a selective ultrafilter $p$ and a  cardinal arithmetic assumption weaker than $GCH$, Castro-Pereira and Tomita classified the torsion groups that admit a countably compact group topology without non-trivial convergent sequences \cite{castro-pereira&tomita2010}. These torsion Abelian groups coincide with
	the torsion Abelian groups that admit a $p$-compact group topology without non-trivial convergent sequences. In particular, it is shown that there exist large countably compact torsion groups without non-trivial convergent sequences.
	
	Comfort called the direct sums of ${\mathbb Q}$ the test space for pseudocompactness for non-torsion groups.	Our aim is to show the result in the abstract. This gives the first example of arbitrarily large torsion free countably compact groups without non-trivial convergent sequences. Also, it follows that ${\mathbb R}$ can be endowed with a $p$-compact group topology without non-trivial convergent sequences whenever $p$ is a selective ultrafilter. 
	
	For the construction in this work, we will adapt the idea of integer stacks in \cite{tomita2015}, used for direct sums of ${\mathbb Z}$'s, to work with direct sums of ${\mathbb Q}$'s. There are some key changes in how the stack is organized and we will give a complete description of rational stacks without assuming knowledge of the previous construction.
	
	We will also use some ideas in \cite{boero&castro-pereira&tomitaoneselective}
	to solve arc equations. 
	
	We notice that this construction uses the fact that the ultrapower of direct sums of ${\mathbb Q}$ is again a ${\mathbb Q}$-vector space. The ultrapower of a free Abelian group is no longer a free Abelian group and as mentioned above, there are no $p$-compact free Abelian groups.

	\subsection{Some basic notation} Fix an infinite cardinal $\kappa$ such that $\kappa=\kappa^\omega$. 
	
	Let $\mathbb T$ be the Abelian group $\mathbb R/\mathbb Z$.
	
	Let $G$ be the Abelian additive group $\mathbb Q^{(\kappa)}=\{g \in \mathbb Q^\kappa: |\supp g|<\omega\}$ and let $H$ be the Abelian additive group $\mathbb Z^{(\kappa)}=\{g \in \mathbb Z^\kappa: |\supp g|<\omega\}$. If $C\subseteq \kappa$, let $\mathbb Q^{(C)}=\{g \in G: \supp g\subseteq C\}$.

	\begin{defin}  
		Given a $\mu \in \kappa$, we denote by $\chi_\mu$  the element of $G$ such that $\supp \chi_\mu = \{\mu\}$ and $\chi_\mu(\mu)=1$. 
		
		Given $\mu \in \kappa$, we define $\vec {\mu}$ the sequence such that $\vec{\mu}(n)=\mu$ for every $n\in \omega$. 
		
		If $A\subseteq \omega$ and $\zeta:\, A \longrightarrow \kappa$ then we define $\chi_\zeta \in G^A$ such that $\chi_\zeta (n)= \chi _{ \zeta(n)}$ for each $n \in \text{dom } \zeta$.
	\end{defin}
	
	\begin{defin} Given $f \in G^A$ and $s=(s_n: n \in A)$ a sequence of rational numbers, we will we denote by $s.f$ or $sf$ the function from $A$ into $G$ such that $(s.f)(n)=s_n.f(n)$, for each $n\in A$.
	\end{defin}
	
	\begin{defin} Given ${\mathcal A} \subseteq G^A$ and $s=(s_n: n \in A)$ a sequence of rational numberss, we define $s\mathcal A=\{sf: f \in \mathcal A\}$. If $s:A\rightarrow \mathbb Q\setminus \{0\}$, we define $\frac{\mathcal A}{s}=\{\frac{1}{s}f: f \in \mathcal A\}$.
	\end{defin}
	
	\smallskip
	
	Given an ultrafilter $q$, we define a equivalence relation on $G^\omega$ by letting $f\simeq_q g$ iff $\{n \in \omega: f(n)=g(n)\} \in q$. We let $[f]_q$ be the equivalence class determined by $f$ and $G^\omega/q$ be $G^\omega/\simeq_q$. Notice that this set has a natural $\mathbb Q$-vector space structure.
	
	A rough idea of  the construction framework is the following:
	\begin{itemize}
		\item 
		
		Define an associated family of stacks (Lemma \ref{Lem.stack.1}).

		\item Solve arc equations in a level of a stack (Lemma \ref{Lem.stack.2}).
		
		\item Find a relation between the arc equations related to a finite sequence of functions and some arc equations in the associated stack (Lemma \ref{Lem.stack.3}).
		
		\item Find the levels of the stacks within an element of the ultrafilter $p$. Solve the arc equations in the stack and build the homomorphism.
		(Proof of Lemma \ref{main.lemma} in Section \ref{proof.of.Main.Lemma}).

	\end{itemize}
	
	\section{Homomorphisms, arc functions and arc equations}

	The way to construct the group topology is through homomorphisms that preserve $p$-compactness for the representatives of a base for the ultrapower $p$. The homomorphisms are constructed by defining smaller and smaller arcs that will give an approximation of the homomorphism. These approximations are arc equations that need to have a solution inside an element of the ultrafilter.
	
	\begin{defin}
		
		Let ${\mathbb B}=\{I+\mathbb Z:I\subset\mathbb R$ is an open interval$\}$ be the family of all the open arcs in ${\mathbb T}$, including ${\mathbb T}$ itself. 
		
		We will call an arc function a function $\phi:\, \kappa \longrightarrow {\mathbb B}$ such that $\{\xi \in \kappa:\, \phi(\xi)\neq {\mathbb T}\}$ is finite. This set will be called the support of the arc function. If the arcs $\phi(\xi)$ have length $\epsilon$ for each $\xi$ in the support of $\phi$, we will call it an $\epsilon$-arc function.
		
		Given two arc functions $\psi$ and $\phi$, we will say that $\psi \leq \phi$ if $\psi(\xi) = \phi(\xi)$ or $\overline {\psi(\xi)}\subseteq \phi(\xi)$, for each $\xi \in \kappa$.
		
		Given an arc function $\phi$ and a positive integer $S$, $S.\phi$ is the arc function with support $\supp \phi$ such that ($S.\phi)(\mu)=S.\phi(\mu)$ for every $\mu \in \kappa$.
	\end{defin}
	
	Below we give the definition of what we will call solutions of an arc equation.
	
	\begin{defin} 
		An arc equation is a quintuple $(\phi, A, \mathcal A, S, U)$ where $\phi$ is an arc function,  $A\subseteq \omega$, ${\mathcal A} \subseteq ({\mathbb Z}^{(\kappa)})^A$, $S$ is a positive integer and $U=(U_f: f \in \mathcal A)$ is a family of elements of ${\mathbb B}$.
		
		Given $n\in A$, an $n$-solution for the arc equation  $(\phi, A, \mathcal A, S, U)$ is an arc function $\psi$ with $S\psi \leq \phi$ such that $\sum_{\mu \in \supp f(n)} f(n)(\mu)\psi(\mu) \subseteq U_f$, for each $f\in {\mathcal A}$. 
	\end{defin}
	
	We will use arc functions as approximations of a homomorphism function in a countable subgroup. The solutions of these equations will depend on the notion of stacks that will be defined later in this work.
	
	The goal is to prove the following:
	\begin{lem}[Main Lemma] 
		Fix a selective ultrafilter $p$.
		Let $\mathcal F\subseteq G^\omega$ be a countable collection of distinct elements mod $p$ such that $\{[f]_p: f \in \mathcal F\} \dot\cup\{ [\chi_{\vec{\mu}}]_p:\, \mu \in \kappa \}$ is $\mathbb Q$-linearly independent in $G^\omega/p$. \label{main.lemma}

		Let $d, d_0,d_1  \in G\setminus \{0\}$ with $\supp d$, $\supp d_0$, $\supp d_1$ pairwise disjoint, and $C$ be a countably infinite subset of $\kappa$ such that $\omega\cup \supp d\cup \supp d_0 \cup \supp d_1\cup \bigcup_{f \in \mathcal F, n \in \omega}\supp f(n) \subseteq C$. For each $f \in \mathcal F$, choose $\xi_f \in C$.
		
		Then there exists a homomorphism $\phi:\, {\mathbb Q}^{(C)}\longrightarrow {\mathbb T}$
		such that
		\begin{enumerate}[label=\alph*)]
			\item $\phi(d)\neq 0$, $\phi(d_0)\neq \phi(d_1)$, and
			\item $p$-$\lim (\phi (\frac{1}{N}.f))=\phi(\frac{1}{N}.\chi_{\xi_f})$, for each $f \in \mathcal F$ and $N\in \omega$.
		\end{enumerate}

	\end{lem}
	For the remaining of this section, let $\{ f_\alpha:\, \omega \leq \alpha < \kappa\}$ be an enumeration of $G^{\omega}$ such that
	
	$$\bigcup_{n\in\omega}\supp f_\xi (n) \subseteq \xi\text{, for each }\xi \in [\omega , \kappa).$$
	By applying the Main Lemma, we get the following result:
	\begin{lem} \label{main.lemma.2}
		Fix a selective ultrafilter $p$.
		Let $I \subseteq [\omega , \kappa)$ be such that $\{ [f_\xi]_p:\, \xi \in I\} \cup \{ [\chi_{\vec{\mu}}]_p:\, \mu \in \kappa \}$ is a $\mathbb Q$-basis for $G^\omega/p$.

		Let $d\in G\setminus \{0\}$, $D\in[I]^{<\omega}\setminus\{\void\}$, $r\in G$ of support $D$ and $B\in p$. Let $C$ be a countably infinite subset of $\kappa$ such that $\omega\cup D \cup \supp d \subseteq C$ and 	 $\bigcup_{n \in \omega}\supp f_\xi(n) \subseteq C$ for every $\xi \in C\cap I$.
		
		Then there exists a homomorphism $\phi:\, {\mathbb Q}^{(C)}\longrightarrow {\mathbb T}$
		such that
		\begin{enumerate}[label=\alph*)]
			\item $\phi(d)\neq 0$,
			\item $p$-$\lim (\phi (\frac{1}{N}.f_\xi))=\phi(\frac{1}{N}.\chi_\xi)$, for each $\xi \in I \cap C$ and $N\in \omega$, and
			\item $(\phi(\sum_{\mu\in D}r(\mu)f_\mu(n)):n\in B)$ does not converge.	
		\end{enumerate}
	\end{lem}
	
	\begin{proof} Let $B' \in p$ a subset of $B$ such that  $(\sum_{\mu\in D}r(\mu)f_\mu(n):n\in B')$ is a 1-1 sequence. Let ${\mathbb A}$ be an almost disjoint family on $B'$ of cardinality ${\mathfrak c}$ and $h_x:\, \omega \longrightarrow \{\sum_{\mu \in D}r(\mu)f_\mu(n):\, n\in x\}$ be a bijection for each $x \in \mathbb A$. \\

		\textbf{Claim}: There exist $x_0,x_1 \in {\mathbb A}$ such that $\{ [f_\xi]_p:\, \xi \in C\cap I\} \cup \{ [\chi_{\vec{\mu}}]_p:\, \mu \in \kappa \}\cup \{[h_{x_0}]_p,[h_{x_1}]_p\} $ is a linearly independent subset.\\
		
		\textbf{Proof of the claim:} Given $x_0, x_1 \in \mathbb A$, notice that $h_{x_1}(n)\neq h_{x_2}(n)$ for all but a finite numbers of $n$'s, so $[h_{x_1}]_p\neq [h_{x_2}]_p$. Since $\mathbb Q$ is countable, it follows that $\langle [h_{x}]_p: x \in \mathbb A\rangle$ has cardinality $\mathfrak c$, so there is $J\subseteq\mathbb A$ such that $|J|=\mathfrak c$ and that $([h_x]_p: x \in J)$ is linearly independent. Now notice that $\langle f_{\xi}: \xi \in I\cap C\rangle \oplus \langle\chi_{\vec\xi}: \xi \in C\rangle$ is countable, so there exists $x_0, x_1 \in J$ such that $\{ [f_\xi]_p:\, \xi \in C\cap I\} \cup \{ [\chi_{\vec{\mu}}]_p:\, \mu \in C \}\cup \{[h_{x_0}]_p,[h_{x_1}]_p\} $ is linearly independent. Since all the supports of these elements are contained in $C$, it is straightforward to see that $\{ [f_\xi]_p:\, \xi \in C\cap I\} \cup \{ [\chi_{\vec{\mu}}]_p:\, \mu \in \kappa \}\cup \{[h_{x_0}]_p,[h_{x_1}]_p\} $ is a linearly independent. \qed\\
		
		Let ${\mathcal F}=\{ f_\xi:\, \xi \in C\cap I\} \cup \{h_{x_0}, h_{x_1}\}$. Set $\xi _f=\mu$ if $f=f_\mu$ for some $\mu \in C\cap I$ and $\xi_f= m_i$ if $f=h_{x_i}$ for some $i<2$ where $m_0\neq m_1$ and $m_0,m_1 \in \omega \setminus \supp d$. Let $d_0=\chi_{m_0}$, $d_1=\chi_{m_1}$ and $\phi$ be as in Lemma \ref{main.lemma}.
		
		Clearly conditions a) and b) of Lemma \ref{main.lemma.2} are satisfied. 
		
		Furthermore, 
		$(\phi(h_{x_i}(k)):k\in \omega)$ has $\phi(\chi_{m_i})$ as an accumulation point for $i<2$. Since these sequences are reorderings of a subsequence of 
		$(\phi(\sum_{\mu\in D}r(\mu)f_\mu(n)):n\in B)$ and  $\phi(\chi_{m_0})\neq \phi(\chi_{m_1})$, it follows that c) is satisfied.
	\end{proof}

	Finally, we may extend the above homomorphism to $G$:
	
	\begin{lem} \label{main.lemma.revisited}
		Fix a selective ultrafilter $p$.
		Let $I \subseteq [\omega , \kappa)$ be such that $\{ [f_\xi]_p:\, \xi \in I\} \cup \{ [\chi_{\vec{\mu}}]_p:\, \mu \in \kappa \}$ is a $\mathbb Q$-basis for $G^\omega/p$.

		Let $d \in G\setminus \{0\}$, $D\in[I]^{<\omega}\setminus\{\void\}$, $r:D\to\mathbb Q\setminus \{0\}$ and $B\in p$.
		
		Then there exists a homomorphism $\phi:\, G\longrightarrow {\mathbb T}$
		such that
		\begin{enumerate}[label=\alph*)]
			\item $\phi(d)\neq 0$,
			\item $p$-$\lim \phi( \frac{1}{N}.f_\xi)=\phi(\frac{1}{N}.\chi_\xi)$, for each $\xi\in I$ and $N\in \omega$, and
			\item $(\phi(\sum_{\mu\in D}r(\mu)f_\mu(n)):n\in B)$ does not converge.
		\end{enumerate}
	\end{lem}
	
	\begin{proof}
		Let $C$ be a countably infinite subset of $\kappa$ such that $\omega \cup \supp d \subseteq C$ and 	 $\bigcup_{n \in \omega}\supp f_\xi(n) \subseteq C$ for every $\xi \in C\cap I$. Such a $C$ exists by standard closing off arguments. Let $(\xi_\alpha: \alpha<\kappa)$ be a strictly increasing enumeration of $\kappa\setminus C$. Let $\phi$ be as in the Main Lemma.
		
		For each $\alpha<\kappa$, let $C_\alpha=C\cup\{\xi_\beta: \beta<\alpha\}$ (so $C_0=C$ and $C_\kappa=\kappa$). Notice that for each $\alpha$ and $n \in \omega$, $\supp f_{\xi_\alpha}(n)\subseteq \xi_\alpha\subseteq C_\alpha$.
		
		Recursively we define homomorphisms $\phi_\alpha:\mathbb Q^{(C_\alpha)}\rightarrow \mathbb T$ for $\alpha\leq \kappa$ satisfying:
		
		\begin{enumerate}[label=\alph*)]
			\item $\phi_0=\phi$,
			\item $\phi_\beta\subseteq \phi_\alpha$ whenever $\beta\leq\alpha\leq\kappa$, and
			\item $p$-$\lim \phi_\alpha (\frac{1}{N}.f_\xi)=\phi(\frac{1}{N}.\chi_\xi)$, for each $\xi\in I\cap C_\alpha$ and $N\in \omega$.
		\end{enumerate}
		
		We let $\phi_0=\phi$. For the limit step, just take unions. For the successor step $\alpha+1$ we proceed as follows:
		
		Notice that $\mathbb Q^{(C_{\alpha+1})}=\mathbb Q^{(C_\alpha)}\oplus \{q\chi_{\xi_\alpha}: q \in \mathbb Q\}$.
		
		First, we define $\tilde \phi_{\alpha}:\{q\chi_{\xi_\alpha}: q \in \mathbb Q\}\rightarrow \mathbb T$ by letting $\tilde\phi_\alpha(\frac{M}{N}\chi_{{\xi_\alpha}})=M(p$-$\lim \phi_{\xi_\alpha} (\frac{1}{N}.f_{\xi_\alpha}))$· Since multiplying a group element by an integer is a continuous function and since $\phi_\alpha$ is a homomorphism, it follows that $\tilde\phi_\alpha$ is well-defined and a group homomorphism. Now let $\phi_{\alpha+1}=\phi_\alpha\oplus\tilde\phi_{\alpha}$.
		
		The required homomorphism is $\phi_\kappa$.
		
	\end{proof}
	
	From this lemma follows the result in the title:
	
	\begin{teo} Assume that $p$ is a selective ultrafilter and $\kappa=\kappa^\omega$ is an infinite cardinal. Then, there exists a $p$-compact group topology on $G={\mathbb Q}^{(\kappa)}$ without non-trivial convergent sequences. 
	\end{teo}
	
	\begin{proof} Let $p$ be a selective ultrafilter, let $I$ be as in the previous theorem. For each $d \in G\setminus \{0\}$,$D\in[I]^{<\omega}\setminus\{\void\}$, $r:D\to\mathbb Q\setminus \{0\}$ and $B\in p$, let $\phi_{d,D,r,B}:G\rightarrow \mathbb T$ be as in the previous theorem.
		
		The group topology induced by these homomorphisms is such that the $p$-$\lim(\frac{1}{N}.f_\xi)=\frac{1}{N}.\chi_\xi$, for each $\xi \in I$ and $N \in \omega$.
		
		If $h$ is any element of $G^\omega$, there exist families $(r_\xi: \xi \in I)$ and $(s_\mu: \mu \in \kappa)$ of rational numbers where all but a finite quantity of them is $0$ such that: 
		
		$[h]_p= \sum_{ \xi \in I} r_\xi [f_\xi]_p+\sum_{\mu \in \kappa}s_\mu [\chi_{\vec{\mu}}]_p$. Then $\sum_{ \xi \in I} r_\xi .\chi_\xi+\sum_{\mu \in \kappa}s_\mu . \chi_\mu$ is the $p$-limit of $h$. Therefore, $G$ is $p$-compact.
		
		To check that there are no non-trivial convergent sequences, fix a 1-1 sequence $g$. Let $(r_\xi: \xi \in I)$ and $(s_\mu: \mu \in \kappa)$ be families of rational numbers where all but a finite quantity of them is $0$ such that: 
		
		$[g]_p= \sum_{ \xi \in I} r_\xi [f_\xi]_p+\sum_{\mu \in \kappa}s_\mu [\chi_{\vec{\mu}}]_p$. Then there are $B\in p$ and $D\in[I]^{<\omega}\setminus\{\void\}$ such that: $g(n)= \sum_{ \xi \in D} r_\xi f_\xi(n)+\sum_{\mu \in \kappa}s_\mu \chi_\mu$, for all $n\in B$.
		By Lemma \ref{main.lemma.revisited}(c), we have that $(\phi_{d,D,r,B}(\sum_{\xi\in D}r_\xi f_\xi(n)):n\in B)$ does not converge in $\mathbb T$, and so $(\sum_{\xi\in D}r_\xi f_\xi(n):n\in B)$ does not converge in $G$. Since 
		$\sum_{\mu \in \kappa}s_\mu \chi_\mu$ is constant, it follows that $(g(n):n\in B)$ does not converge,and therefore $g$ does not converge.
	\end{proof}

	\section{A preliminar discussion on rational stacks}

	\medskip
	
	We will start this section with an informal discussion about rational stacks. A rational stack will be a nonuple $\langle   \mathcal B, \nu, \zeta, K, A, k_0, k_1, l, T\rangle$, where:
	\begin{itemize}
		\item $A\subseteq \omega$ is infinite,
		\item $k_0\leq k_1$ are natural numbers with $k_1>0$,
		\item $l:k_1\rightarrow \omega$,
		\item $\nu:k_0\rightarrow \kappa$,
		\item $\zeta:k_1\rightarrow \kappa^\omega$,
		\item $K:A\rightarrow \omega\setminus 2$ is such that for every $n \in A$, $n!T\mid K_n$,
		\item $\mathcal B=({{\mathcal B}}_{i,j} :\, i< k_1, j < l_i)$ is such that each $\mathcal B_{i, j}\subseteq H^\omega$ is finite.
		\item $T>0$ is an integer.
	\end{itemize}
	
	In order to be a rational stack, this nonuple must satisfy additional properties. The full definition of rational stack will be given in Section 5. This definition was designed to solve arc equations to construct the homomorphisms we want.

	Roughly speaking, we associate each finite family of sequences to a stack. We transform the arc equations associated to this finite family to arc equations associated to the stack, solve the arc equation using the properties of the stack then return to a solution of the original arc equations. We want to solve infinitely many equations, thus, this process is made back and forth, including more equations at each step. At each stage, the stack is different  and there is no containment relation between the stacks, even though we use a larger finite subfamily of sequences.
	
	The idea is that before each iteration we have arc equations related to a certain arc size. We might need to shrink the arcs when we write the arc equations related to the stack so that the solution of arc equation associated to the stack can be transformed into a solution of the original equations. After each step, the output arcs will be smaller and will be the input arcs for the next iteration.  We have to estimate the size of the output arcs according to its input arc so that the sets of arc equations can be solved ultrafilter often. For this, we use the happiness of the selective ultrafilter.
	
	Before we even define the stack, we will list the main results that motivated its definition.

	The Lemma below associates each finite subset of functions to a stack that will be used to solve arc equations associated to this family. 
	
	\begin{lem} \label{Lem.stack.1}
		Let $B \in p$ and $\mathcal G$ be a finite subset of $G^\omega$ whose elements are distinct mod $p$ and none of them is constant mod $p$ such that $\{[f]_p:f\in \mathcal G\}\cup\{[\chi_{\vec\nu}]_p:\nu\in\kappa\}$ is linearly independent. Then there exists a rational stack  $\mathcal S=\langle  \mathcal B, \nu, \zeta, K, A, k_0, k_1, l, T\rangle$ such that, by defining
		${\mathcal A}={\mathcal G} \cup \{ \chi_{\vec{\nu_i}}:\, i<k_0\}$ and ${\mathcal C}= \frac{\bigcup_{i< k_1, j<l_i} {\mathcal B}_{i,j}}{{K}}$, there exist  ${\mathcal M}:\mathcal A\times \mathcal C\rightarrow \mathbb Z$, ${\mathcal N}:\mathcal C\times \mathcal A\rightarrow \mathbb Z$ satisfying:
		
		\begin{enumerate}[label=(\arabic*)]
			\item $\{ [f]_p:\, f \in {\mathcal A} \}$ and $\{ [h]_p:\, h \in {\mathcal C} \}$ generate the same subspace;
			\item $f(n) =\sum_{h \in {\mathcal C}}{\mathcal M}_{f,h}h(n)$, for each $n\in A$ and $f \in {\mathcal A}$,
			\item $h(n) = \frac{1}{T^2}.\sum_{f\in{\mathcal A}}{\mathcal N}_{h,f}f(n)$, for each $n\in A$ and $h\in{\mathcal C}$,
			
			\item $K\mathcal A\subseteq H^\omega$,
			\item $K\mathcal C\subseteq H^\omega$, and
			\item $A\subseteq B$.
		\end{enumerate}
	\end{lem}
	
	\begin{proof}
		The proof is quite technical and will be presented in a later section.
	\end{proof}
	
	Notice that if we interpret $\mathcal M$ and $\mathcal N$ as matrices, then $\mathcal M$ and $\frac{1}{T^2}\mathcal N$ are inverse matrices.
	
	\begin{lem} \label{Lem.stack.2} Let $\mathcal S$, ${\mathcal A}$, ${\mathcal C}$, ${\mathcal M}$ and ${\mathcal N}$ be as in Lemma  \ref{Lem.stack.1}. Let $\epsilon$ be a positive real and $D$ be a finite subset of $\kappa$. Then there exist $B \subseteq A$ cofinite in $A$ and a family of positive real numbers $(\gamma_n: n \in B)$ such that:
		
		For every $n \in B$, for every family  $( W_h:\, h \in {\mathcal C} )$ of open arcs of length $\epsilon$, and for every arc function $\psi$ of length $\epsilon$ such that $\supp \psi \subseteq D\setminus \{\nu_0,\ldots \nu_{k_0-1}\}$, there exists an $n$-solution of length $\gamma_n$ for the arc equation $(\psi,B,K.\mathcal C,K_n,W)$. 
	\end{lem}
	
	\begin{proof} The proof is quite technical and will be presented in a later section.
	\end{proof}

	\begin{lem} \label{Lem.stack.3} Let $\mathcal S$, ${\mathcal A}$, ${\mathcal C}$, ${\mathcal M}$ and ${\mathcal N}$ be as in Lemma  \ref{Lem.stack.1}. Let $\delta$ be a positive real such that $\epsilon =\frac{\delta}{\sum_{f\in {\mathcal A}, h\in{\mathcal C} }|{\mathcal M}_{f,h}|}< 1$.
		
		Let $(U_f:\, f\in {\mathcal A})$ be a family of open arcs of length $\delta$. Let  $\varrho$  be an arc function of length $\delta$ such that $U_{\chi_{\vec{\nu}_i}}=\varrho(\nu_i)$ for $i<k_0$.
		Furthermore, assume that $\{\nu_i:\, i < k_0\} \subseteq \supp \varrho$.
		
		Then there exist $(W_h:\, h \in {\mathcal C})$ a family of open arcs of length $\epsilon$  and $\psi$ an $\epsilon$-arc function with  support  $\supp \varrho \setminus \{\nu_0, \ldots, \nu_{k_0-1}\}$ such that for every $n\in A$, every $n$-solution for the arc equation $(\psi,A,K.\mathcal C,K_n,W)$ is an $n$-solution for $(\varrho,A,K.\mathcal A,K_n,U)$.
		
	\end{lem}
	
	\begin{proof}

		% Let $\psi(\mu)$ be an arc of length $\epsilon$ whose center is $\frac{x_\mu}{N_{n_b}}$.
		
		Given $f \in \mathcal A$, let $y_f \in {\mathbb R}$ be such that $y_f+\mathbb Z$ is the center of the arc $U_f$.

		For each $h \in \mathcal C$, let $z_h= \sum_{f \in {\mathcal A}} {\mathcal N}_{h,f}.\frac{y_f}{T^2}$.
		Since ${\mathcal N}$ is an integer matrix, it follows that $z_h+{\mathbb Z}= \sum_{f \in {\mathcal A}} {\mathcal N}_{h,f}.(\frac{y_f}{T^2}+{\mathbb Z})$. Let $W_h$ be an arc centered in $z_h+{\mathbb Z}$ whose length is $\epsilon$.

		Let $\psi(\mu)$ be an arc of same center as $\varrho(\mu)$ of length $\epsilon$ for each $\mu \in \supp \varrho \setminus \{ \nu_0, \ldots , \nu_{k_0-1}\}$ and $\psi(\nu_i)=\mathbb T$ for each $i<k_0$.

		Suppose $\phi$ is an $n$-solution for $(\psi,A,K.\mathcal C,K_n,W)$.  Then $\sum_{\mu \in \supp h(n)}K_n.h(n)(\mu).\phi(\mu)\subseteq W_h$, for each $h \in {\mathcal C}$. Also, we have that, for each $\mu \in \kappa \setminus \{ \nu_0,\ldots, \nu_{k_0-1}\}$:
				\begin{equation}\label{perhaps} \tag{\#} K_n.\phi (\mu) \leq \psi (\mu)\leq \varrho (\mu)
			\end{equation}

		Let $f\in\mathcal A$.
		
		We have that for each $\mu$, $\sum_{h \in {\mathcal C}}{\mathcal M}_{f,h}.K_n.h(n)(\mu).\phi(\mu)=K_n.f(n)(\mu)\phi(\mu)$. Therefore,	$\sum_{h \in {\mathcal C}}{\mathcal M}_{f,h}.\sum_{\mu \in \supp h(n)}K_n.h(n)(\mu).\phi(\mu)=\sum_{\mu \in \supp f(n)}K_n.f(n)(\mu).\phi(\mu)$.
		
		It then follows that $\sum_{\mu \in \supp f(n)}K_n.f(n)(\mu).\phi(\mu) \subseteq \sum_{h \in {\mathcal C}} {\mathcal M}_{f,h}.W_h$. The arc $\sum_{h \in {\mathcal C}}{\mathcal M}_{f,h}.W_h$ is centered in $\sum_{h\in {\mathcal C}}{\mathcal M}_{f,h}.(z_h+\mathbb Z)= \sum_{h\in{\mathcal C}}{\mathcal M}_{f,h}\sum_{g \in {\mathcal A}} {\mathcal N}_{h,g}.(\frac{y_g}{T^2}+{\mathbb Z})=\sum_{g \in {\mathcal A}}\sum_{h \in {\mathcal C}}{\mathcal M}_{f,h} {\mathcal N}_{h,g}.(\frac{y_g}{T^2}+{\mathbb Z})=y_f+{\mathbb Z}$, and has length $\epsilon \cdot\sum_{h\in {\mathcal C}}|{\mathcal M}_{f,h}| \leq \delta$. Therefore,  
		\begin{equation}\label{permaps} \tag{*} \sum_{h\in{\mathcal C}}{\mathcal M}_{f,h}.W_h \subseteq U_f.
		\end{equation}
		Thus, $\phi$ is an $n$-solution for $(\varrho,A,K.\mathcal A,K_n,U)$,  as required, provided that we show  $K_n.\phi \leq \varrho$.
		
		From $\eqref{permaps}$, if $f=\chi_{ \vec{\nu_i}}$ then $K_n.\phi(\nu_i)=K_n \chi_{ \vec{\nu_i}}(n)(\nu_i) \phi(\nu_i) \subseteq U_f=\varrho(\nu_i)$, hence  $K_n.\phi (\nu_i)\leq \varrho(\nu_i)$ for each $0\leq i <k_0$. This and \eqref{perhaps} implies that $K_n\phi \leq \varrho$.

	\end{proof}

	\section{Selective ultrafilters and the proof of the Main Lemma using the properties of the stacks} \label{proof.of.Main.Lemma}

	Our goal in this section is  to prove Lemma \ref{main.lemma} using the Lemmas in the previous section. First we state the following Lemma:
	
	\begin{lem}
		Fix a selective ultrafilter $p$.
		Let $\mathcal F\subseteq G^\omega$ be a countable collection of distinct elements mod $p$ such that $\{[f]_p: f \in \mathcal F\} \dot\cup\{ [\chi_{\vec{\mu}}]_p:\, \mu \in \kappa \}$ is $\mathbb Q$-linearly independent in $G^\omega/p$.

		Let $d, d_0,d_1  \in G\setminus \{0\}$ with $\supp d$, $\supp d_0$, $\supp d_1$ pairwise disjoint, and $C$ be a countably infinite subset of $\kappa$ such that $\omega\cup \supp d\cup \supp d_0 \cup \supp d_1\cup \bigcup_{f \in \mathcal F, n \in \omega}\supp f(n) \subseteq C$. For each $f \in \mathcal F$, choose $\xi_f \in C$. Let $(\mathcal F^n:n\in\omega)$ be an increasing sequence of finite sets whose union is $\mathcal F$.
		
		Then there exist stacks $\mathcal S^{m}=\langle  \mathcal B^{m}, \nu^{m}, \zeta^{m}, K^{m}, A^{m}, k_{0}^{m}, k_1^{m}, l^{m}, T^m\rangle$ and ${\mathcal A}^{m}$, ${\mathcal C}^{m}$, $\mathcal M^m$, $\mathcal N^m$ related to the stacks as in Lemma \ref{Lem.stack.1}, $r:\omega\cup\{-1\}\rightarrow \omega$ such that $r[\omega] \in p$ and $r(-1)=0$, and a sequence of arc functions $(\varrho^{r(m)}: m\ge-1)$ with $C\subseteq \bigcup_{m\ge-1} \supp \varrho^{r(m)}$, satisfying the following:
		
		\begin{enumerate}[label=\alph*)]
			\item $0\notin \overline{\sum_{\mu \in \supp d}d(\mu).\varrho^0(\mu)}$ and $\overline{\sum_{\mu \in \supp d_0}d_0(\mu).\varrho^0(\mu)} \cap \overline{\sum_{\mu \in \supp d_1}d_1(\mu).\varrho^0(\mu)}=\void$.
			\item For every $m', m\geq -1$ with $m'\leq m$ and for every $\xi \in \supp \varrho^{r(m+1)}$, we have $\left(\prod_{i=m'}^m K^{r(i)}_{r(i+1)}\right).\varrho^{r(m+1)}(\xi) \subseteq \varrho^{r(m')}(\xi)$ and $\left(\prod_{i=m'}^m K^{r(i)}_{r(i+1)}\right).\varrho^{r(m+1)}(\xi)$ has length $\leq \frac{1}{2^{r(m)+1}}$. 
			\item For every $m \geq -1$ and $f \in \mathcal F^{r(m)}$, we have $\sum_{\mu \in \supp f(r(m+1))} \ f(r(m+1))(\mu)K^{r(m)}_{r(m+1)}\varrho^{r(m+1)}(\mu) \subseteq \varrho^{r(m)}(\xi_f)$. 
			\item For every $m, m' \geq -1$ with $m'\leq m$ and $f \in \mathcal F^{r(m)}$, we have $\sum_{\mu \in \supp f(r(m+1))} \ f(r(m+1))(\mu)\left(\prod_{i=m'}^m K^{r(i)}_{r(i+1)}\right)\varrho^{r(m+1)}(\mu) \subseteq \varrho^{r(m')}(\xi_f)$.
			\item For every $m\geq -1$, $\supp \varrho^{r(m)}\subseteq \supp \varrho^{r(m+1)}$.
			
		\end{enumerate}

		\begin{proof}
			
			Let $p$, $\mathcal F$, $d$, $d_0$, $d_1$,$C$ and $(\xi_f: f \in \mathcal F)$ be given. Write $C$ as an increasing sequence of finite sets $(C^n:n\in\omega)$ such that for each $n\in\omega$, $\bigcup\{\supp f(k): f\in\mathcal{F}^n\text{ and }k\le n\}\subset C^n$, and $\supp d\cup \supp d_0 \cup \supp d_1\subset C^0$. 
			
			Apply Lemma \ref{Lem.stack.1} to $\mathcal G=\mathcal F^0$ and $B=\omega$ to obtain a rational stack $\mathcal S^0=\langle  \mathcal B^0, \nu^0, \zeta^0, K^0, A^0, k_0^0, k_1^0, l^0, T^0\rangle$ and ${\mathcal A}^0$, ${\mathcal C}^0$,  ${\mathcal M}^0$ and ${\mathcal N}^0$ satisfying (1)-(7) as in the Lemma.
			
			Fix $\delta^{0}\in\mathbb R$ such that $0<\delta^0<1$ and $\varrho^0$ a $\delta^{0}$-arc function such that $0\notin \overline{\sum_{\mu \in \supp d}d(\mu).\varrho^0(\mu)}$ and $\overline{\sum_{\mu \in \supp d_0}d_0(\mu).\varrho^0(\mu)} \cap \overline{\sum_{\mu \in \supp d_1}d_1(\mu).\varrho^0(\mu)}=\void$. We will also assume that $C^0\cup \{\nu_i^0:i<k^0_0\} \subseteq \supp \varrho^0$.
			
			Let $\epsilon^0=\frac{\delta^0}{\sum_{f \in {\mathcal A^0}, h \in {\mathcal B^0}}|{\mathcal M}^0_{f,h}|}$. Notice that with this $\epsilon^0$ we may apply Lemma \ref{Lem.stack.3}.
			
			Now we apply Lemma \ref{Lem.stack.2} with $D=C_0$ to obtain $B^0 \subseteq A^0\setminus 1$ and $\boldsymbol{\gamma}^0=(\gamma^0_n: n \in B^0)$ as in the Lemma.

			Suppose that $B^t \in p$  with $(B^t:\, t \leq m)$  decreasing family of subsets of $\omega$, $\gamma^t_n$ for $t\leq m$ and $n \in B^t$ are defined. 
			
			Define $\delta^{m+1} =\frac{1}{2^{m+2}}. \frac{1}{ \prod_{i,n\leq m+1} K^i_n}.\min(\{ \gamma_{n}^t:\, t< n \leq m+2, n \in B^t  \}\cup \{1\})$.
			
			Apply Lemma \ref{Lem.stack.1} with $\mathcal G=\mathcal F^{m+1}$ and $B=B^m$ to obtain a stack $\mathcal S^{m+1}=\langle  \mathcal B^{m+1}, \nu^{m+1}, \zeta^{m+1}, K^{m+1}, A^{m+1}, k_{0}^{m+1}, k_1^{m+1}, l^{m+1}, T^{m+1}\rangle$ and ${\mathcal A}^{m+1}$, ${\mathcal C}^{m+1}, \mathcal M^{m+1}, \mathcal N^{m+1}$ related to the stack as in the lemma. Then $A^{m+1} \subseteq B^m$. Let $\epsilon^{m+1}=\frac{\delta^{m+1}}{\sum_{f \in {\mathcal A}^{m+1},h \in {\mathcal B}^{m+1}}|{\mathcal M}^{m+1}_{f,h}|}$. Notice that with this $\epsilon^{m+1}$ we may apply Lemma \ref{Lem.stack.3}.  
			
			Now we apply Lemma \ref{Lem.stack.2} with $D=C_{m+1}$ using to obtain $B^{m+1} \subseteq A^{m+1}\setminus m+2$ and $\boldsymbol{\gamma}^{m+1}=(\gamma^{m+1}_n: n \in B^{m+1})$ as in the Lemma.

			We will use the happiness of the selective ultrafilter $p$: the sets constructed previously  $B^0 \supseteq B^1 \ldots$ are all elements of $p$, so there exists a function $r \in \omega^\omega$ such that $r[\omega]\in p$, $r(0)\in B^0$ and, for all $n\in \omega$, $r(n+1)\in B^{r(n)}$.

			Define $U^0=(U^0_f: f \in \mathcal A^0)$, where $U^0_f=\varrho^{0}(\xi_f)$ if $f\in \mathcal{F}^0$ or $U^0_f=\varrho^0(\nu^0_i)$ if $f=\chi_{\vec{\nu}^0_i}$. By Lemma \ref{Lem.stack.3} applied on stage 0, we obtain $\psi$ and $W$ as in the conclusion of Lemma \ref{Lem.stack.3}. Now, according to the conclusion of Lemma \ref{Lem.stack.2} used in stage 0 of the construction, since $r(0)\in B^0$, we obtain a $r(0)$-solution $\phi^0$ of length $\gamma^0_{r(0)}$ to the arc equation $(\psi,B^0,K^0.\mathcal C^0,K^0_{r(0)},W)$.
			
			Now, using the conclusion of Lemma \ref{Lem.stack.3}, we have that $\phi^0$ is a $r(0)$-solution to $(\varrho^0,B^0,K^0.\mathcal A^0,K^0_{r(0)},U^0)$.
			
			In particular,  we have, for each $f\in\mathcal F^0$,
			
			$\sum_{\mu \in \supp f(r(0))} \ f(r(0))(\mu)K^0_{r(0)}\phi^0(\mu) \subseteq \varrho^0(\xi_f)$.
			
			Let $r(-1)=0$.  We can inductively to construct $\phi^{r(m)}$, $\varrho^{r(m)}$ a $\delta^{r(m)}$-arc function with $C^{r(m)} \subseteq \supp \varrho^{r(m)}$, and $U^{r(m)}$ such that:
			\begin{enumerate}[label=\arabic*)]
				\item For every $m\geq -1$, $\phi^{r(m)}$ is a $r(m+1)$-solution of length $\gamma^{r(m)}_{r(m+1)}$ of the arc equation $(\varrho^{r(m)}, B^{r(m)}, K^{r(m)}.\mathcal A^{r(m)}, K^{r(m)}_{r(m+1)}, U^{r(m)})$, 
				\item For every $m \geq -1$, $\varrho^{r(m+1)}\leq \phi^{r(m)}$,
				\item For every $m', m\geq -1$ with $m'\leq m$ and for every $\xi \in \supp \varrho^{r(m+1)}$, we have $\left(\prod_{i=m'}^mK^{r(i)}_{r(i+1)}\right).\varrho^{r(m+1)}(\xi) \subseteq \varrho^{r(m')}(\xi)$ and $\left(\prod_{i=m'}^mK^{r(i)}_{r(i+1)}\right).\varrho^{r(m+1)}(\xi)$ has length $\leq \frac{1}{2^{r(m)+1}}$, 
				\item For every $m \geq -1$ and $f \in \mathcal F^{r(m)}$, we have $\sum_{\mu \in \supp f(r(m+1))} \ f(r(m+1))(\mu)K^{r(m)}_{r(m+1)}\varrho^{r(m+1)}(\mu) \subseteq \varrho^{r(m)}(\xi_f)$,
				\item For every $m, m' \geq -1$ with $m'\leq m$ and $f \in \mathcal F^{r(m)}$, we have $\sum_{\mu \in \supp f(r(m+1))} \ f(r(m+1))(\mu)\left(\prod_{i=m'}^mK^{r(i)}_{r(i+1)}\right)\varrho^{r(m+1)}(\mu) \subseteq \varrho^{r(m')}(\xi_f)$,
				\item For every $m\geq -1$, $\supp \varrho^{r(m)}\subseteq \supp \varrho^{r(m+1)}$, and
				\item $U^{r(m)}=(U^{r(m)}_f: f \in \mathcal A^{r(m)})$, where $U^{r(m)}_f=\varrho^{r(m)}(\xi_f)$ if $f\in \mathcal{F}^{r(m)}$ or $U^{r(m)}_f=\varrho^{r(m)}(\nu^{r(m)}_i)$ if $f=\chi_{\vec{\nu}^{r(m)}_i}$.
				
			\end{enumerate}

			The base of the recursion is already done. Suppose the construction is done until step $m$ and let us define $\varrho^{r(m+1)}$ and $\phi^{r(m+1)}$.
			
			Let $\varrho^{r(m+1)}$ be a $\delta^{r(m+1)}$-arc funtion such that $\supp \varrho^{r(m)}\cup C^{r(m+1)}\cup\{\nu_i^{r(m+1)}:i<k^{r(m+1)}_0\}\subseteq\supp\varrho^{r(m+1)}$ and $\varrho^{r(m+1)}\le\phi^{r(m)}$.
			
			Now define $U^{r(m+1)}=(U^{r(m+1)}_f: f \in \mathcal A^{r(m+1)})$, where $U^{r(m+1)}_f=\varrho^{r(m+1)}(\xi_f)$ if $f\in \mathcal{F}^{r(m+1)}$ or $U^{r(m+1)}_f=\varrho^{r(m+1)}(\nu^{r(m+1)}_i)$ if $f=\chi_{\vec{\nu}^{r(m+1)}_i}$. By Lemma \ref{Lem.stack.3} applied on stage $m+1$, we obtain $\psi$ and $W$ as in the conclusion of Lemma \ref{Lem.stack.3}. Now, according to the conclusion of Lemma \ref{Lem.stack.2} used in stage $m+1$ of the construction, since $r(m+2)\in B^{r(m+1)}$, we obtain a $r(m+2)$-solution $\phi^{r(m+1)}$ of length $\gamma^{r(m+1)}_{r(m+2)}$ to the arc equation $(\psi,B^{r(m+1)},K^{r(m+1)}.\mathcal C^{r(m+1)},K^{r(m+1)}_{r(m+2)},W)$.
			
			Now, using the conclusion of Lemma \ref{Lem.stack.3}, we have that $\phi^{r(m+1)}$ is a $r(m+2)$-solution to $(\varrho^{r(m+1)},B^{r(m+1)},K^{r(m+1)}.\mathcal A^{r(m+1)},K^{r(m+1)}_{r(m+2)},U^{r(m+1)})$.
			
			With , $\varrho^{r(m+1)}$ and $\phi^{r(m+1  )}$ thus defined, items 1), 2), 6) and 7) of the recursion are immediately satisfied.
			
			In order to verify item 3): the second statement follows from the definition of $\delta^{r(m+1)}$. As for the first, use items 1) and 2) and then use item 3) iteratively.
			
			Item 4) follows from items 1) and 2) and the definiton of $U^{r(m)}$.
			
			Item 5) follows from multiplying the expression in 4) by $\left(\prod_{i=m'}^{m-1}K^{r(i)}_{r(i+1)}\right)$ and then applying item 3) for $m'$ and $m-1$.
			
			Now that the recursion is complete, notice that items a)-e) of the statement of the Lemma are clearly satisfied.

		\end{proof}

	\end{lem}
	
	\begin{lemnumber}[Main Lemma]
		Fix a selective ultrafilter $p$.
		Let $\mathcal F\subseteq G^\omega$ be a countable collection of distinct elements mod $p$ such that $\{[f]_p: f \in \mathcal F\} \dot\cup\{ [\chi_{\vec{\mu}}]_p:\, \mu \in \kappa \}$ is $\mathbb Q$-linearly independent in $G^\omega/p$.

		Let $d, d_0,d_1  \in G\setminus \{0\}$ with $\supp d$, $\supp d_0$, $\supp d_1$ pairwise disjoint, and $C$ be a countably infinite subset of $\kappa$ such that $\omega\cup \supp d\cup \supp d_0 \cup \supp d_1\cup \bigcup_{f \in \mathcal F, n \in \omega}\supp f(n) \subseteq C$. For each $f \in \mathcal F$, choose $\xi_f \in C$.
		
		Then there exists a homomorphism $\phi:\, {\mathbb Q}^{(C)}\longrightarrow {\mathbb T}$
		such that
		\begin{enumerate}[label=\alph*)]
			\item $\phi(d)\neq 0$, $\phi(d_0)\neq \phi(d_1)$, and
			\item $p$-$\lim (\phi (\frac{1}{P}.f))=\phi(\frac{1}{P}.\chi_{\xi_f})$, for each $f \in \mathcal F$ and $P\in \omega\setminus\{0\}$.
		\end{enumerate}
	\end{lemnumber}
	\begin{proof}
		Let $\mathcal S^m$, $\mathcal A^m$, $\mathcal C^m$, $\mathcal M^m$, $\mathcal N^m$, $\mathcal F^m$ and $\varrho^{r(m)}$, $(m \in \omega)$, and $\varrho^{0}$ be as in the previous lemma.
		
		For each $m\in \omega$, let $Q_m=\prod_{i=-1}^{m-1}K^{r(i)}_{r(i+1)}$. Now given a positive integer $m'$ and $\xi \in C\cap \supp \varrho^{r(m')}$, define $\phi\left(\frac{1}{Q_{m'}}.\chi_\xi\right)$ as the unique element of $\bigcap_{m \geq m'} \frac{Q_m}{Q_{m'}} \varrho^{r(m)}(\xi)$. Furthermore, if $P$ divides $Q_{m'}$ then define $\phi(\frac{1}{P}.\chi_\xi)=\frac{ Q_{m'}}{P}\phi\left(\frac{1}{Q_{m'}}.\chi_\xi\right)$. Then, since $n!|K_n^m$ for every $n, m$, $\phi\left(\frac{1}{P}.\chi_\xi\right)$ is well-defined and does not depend on $Q_{m'}$ and $\phi$ can be extended to a homomorphism.
		
		Notice that since $0\notin \overline{\sum_{\mu \in \supp d}d(\mu).\varrho^0(\mu)}$ and
		$$\overline{\sum_{\mu \in \supp d_0}d_0(\mu).\varrho^0(\mu)}\cap \overline{\sum_{\mu \in \supp d_1}d_1(\mu).\varrho^0(\mu)}=\void,$$ it follows that $\phi(d)\neq 0$ and $\phi(d_0)\neq \phi(d_1)$.

		\medskip
		
		Let $f \in \mathcal F$ and $P$ be a positive integer. Let $M$ be positive such that $f \in \mathcal F^M$.
		
		\medskip
		
		{\bf Claim}: $( \phi(\frac{1}{P}.f(r(m))):\, m \in \omega)$ converges to $\phi(\frac{1}{P}.\chi_{\xi_f})$.
		
		\begin{proof}Let $m\geq M$ be such that $P$ divides $Q_{m-1}$ and $\xi_f \in C\cap\supp\varrho^{r(m-1)}$.  Then
			
			$\phi(\frac{1}{P}.f(r(m))) =\phi(\frac{1}{P}.\sum_{\mu \in \supp f(r(m))} f(r(m))(\mu)\chi_\mu)=$
			
			$\phi(\sum_{\mu \in \supp f(r(m))} \frac{f(r(m))(\mu)}{P}.\chi_\mu)=$
			$\sum_{\mu \in \supp f(r(m))} \phi(\frac{f(r(m))(\mu)}{P}.\chi_\mu)=$
			
			$\sum_{\mu \in \supp f(r(m))} \frac{f(r(m))(\mu)}{P}.Q_m.\phi(\frac{1}{Q_m}.\chi_\mu)=$
			
			$\frac{1}{P}.Q_{m-1}\sum_{\mu \in \supp f(r(m))} K^{r(m-1)}_{r(m)} f(r(m))(\mu).\phi(\frac{1}{Q_m}.\chi_\mu)\in $
			
			$\frac{1}{P}.Q_{m-1}\sum_{\mu \in \supp f(r(m))} K^{r(m-1)}_{r(m)} f(r(m))(\mu).\varrho_{r(m)}(\mu) \subseteq $
			
			$\frac{1}{P}.Q_{m-1}\varrho_{r(m-1)}  (\xi_f)$.
			
			This last set is a neighborhood of 
			$\phi(\frac{1}{P}.\chi_{\xi_f})$ and has length at most $\frac{1}{2^{r(m-2)}+1}$. 
			
			This proves the claim.
		\end{proof}
		
		Since $r[\omega] \in p$ it follows that the $p$-limit of  $(\phi(\frac{1}{P}.f(n)):\, n \in \omega)$ is $\phi(\frac{1}{P}.\chi_{\xi_f})$.\end{proof}
	\section{Defining rational stacks}

	We define stacks as a tool to solve a system of arc equations. The way the rational stack is constructed is motivated by the construction in \cite{tomita2015}. 
	
	We want to solve arc equations related to the representatives of the basis for $G^\omega/p$. We construct a stack, associate the original arc equations to arc equations for the stack, solve the arc equations for the stack and these solutions will lead to a solution to the original system of arc equations.
	
	For the most basic example, if we take a  stack with a single element $\langle \{h\} , \{Z_n:\, n \in A \}\rangle$, we choose a point $\zeta(n)$ (condition iii)) in the support of $h(n)$. The values of 
	$h(n)(\zeta(n))$ and $Z_n$ and their ratio (condition v)) are the main ingredients to solve the arc equation. According to the size of the arc entry, either the denominator $Z_n$ or the numerator $h(n)(\zeta(n))$ must be large enough (condition v)).
	The output arc will shrink proportionally to the sizes of the numerator $Z_n$ and the denominator $h(n)(\zeta(n))$.
	
	If we are dealing with a stack with more elements, the same $\zeta$ might have to be used for different sequences. Those $h(n)(\zeta(n))$ that increase at a proportional  speed (condition vi)) will be solved together. But there may have other sequences using the same $\zeta$ that have different speeds.
	
	The arc equations of the bricks whose denominator $h(n)(\zeta(n))$ is smaller will be solved first, since they require larger input arcs.
	To continue using the same $\zeta$ for arc equations for the next brick with the same $\zeta$, the $h(n)(\zeta(n))$ must be much larger to compensate that the output arcs have shrunk to solve the equations for the previous brick (condition vii)).
	
	The sequences may use a different element of the support, but the shrinking of the arcs that solved previous equations must not interfere with the size of the arc related to a different point in the support. For this reason we need conditions ii), iv) and vii). 
	
	The idea to use the stacks to solve arc equations back and forth is based on the idea  in \cite{boero&castro-pereira&tomitaoneselective}.

	The order used in the stack corresponds to the order we construct a  stack associated to a finite number of elements of the representatives of a basis. The order to solve the arc equations is the reverse order of the stack.

	\begin{defin} A rational stack is an nonuple $\langle   \mathcal B, \nu, \zeta, K, A, k_0, k_1, l, T\rangle$, where:
		\begin{itemize}
			\item $A\subseteq \omega$ is infinite,
			\item $k_0\leq k_1$ are natural numbers with $k_1>0$,
			\item $l:k_1\rightarrow \omega$,
			\item $\nu:k_0\rightarrow \kappa$,
			\item $\zeta:k_1\rightarrow \kappa^\omega$,
			\item $K:\omega\rightarrow \omega\setminus 2$ is such that for every $n \in A$, $n!T\mid K_n$,
			\item $\mathcal B=({{\mathcal B}}_{i,j} :\, i< k_1, j < l_i)$ is such that each $\mathcal B_{i, j}\subseteq H^\omega$ is finite,
			\item $T>0$ is an integer.
		\end{itemize}
		
		\begin{enumerate}[label=\roman*)]
			\item $\zeta_i(n)=\nu_i$ for every $i<k_0$ and $n \in A$,
			\item  The elements $\nu_i\, (i<k_0)$ and $\zeta_j(n) \, (k_0\leq j<k_1, n \in A)$ are pairwise distinct,
			\item $\zeta_i(n) \in \supp h(n)$, for each $ i<k_1$, $j<l_i$, $h \in {\mathcal B}_{i,j}$ and $n\in A$,
			\item $\zeta_{i}(n) \notin \supp h(n)$, for each, $i<i_*<k_1$, $j< l_{i_*}$ and $h \in {\mathcal B}_{i_*,j}$ and $n \in A$,
			\item $\left( \frac{h(n)(\zeta_i(n))}{K_n}\right)_{n\in A}$ converges monotonically to $+\infty$, $-\infty$ or a real number, for each $i< k_1$, $j<l_i$ and $h\in {\mathcal B}_{i,j}$,
			\item For every $i<k_1, j<l_i$, there exists $h_* \in \mathcal B_{i, j}$ such that for every $h \in \mathcal B_{i, j}$, $\left(\frac{h(n)(\zeta_i(n))}{h_{*}(n)(\zeta_i(n))}\right)_{n\in A}$ converges to a real number $\theta_{h_*}^h$ and $(\theta_{h_*}^h:\,h \in {\mathcal B}_{i,j})$ is linearly independent (as a $\mathbb Q$-vector space),
			\item For each $i<k_1$, $j'<j<l_i$, $h \in {\mathcal B}_{i,j}$ and $h'\in {\mathcal B}_{i,j'}$, $\left(\frac{h(n)(\zeta_i(n))}{h'(n)(\zeta_i(n))}\right)_{n\in A}$ converges monotonically to $0$,
			\item For each $ i<k_0$ there exists $j<l_i$ such that $\frac{K}{T} . \chi_{\vec{\nu_i}}\in{\mathcal B}_{i,j}$,
			\item $\left(|h(n)(\zeta_i(n))|\right)_{n\in A}$ is strictly increasing, for each $i<k_1, j<l_i$ and $h \in {\mathcal B}_{i,j}$, and
			\item  For each $i<k_1, j<l_i$ and distinct $h, h_* \in {\mathcal B}_{i,j}$, either 
			
			\begin{itemize}
				\item $|h(n)(\zeta_i(n))|> |h_*(n)(\zeta_i(n))|$  for each  $n \in A$, or
				\item   $|h(n)(\zeta_i(n))|= |h_*(n)(\zeta_i(n))|$  for each  $n \in A$, or
				\item  $|h(n)(\zeta_i(n))|< |h_*(n)(\zeta_i(n))|$  for each  $n \in A$.
			\end{itemize}
			
			\item for all $\mu \in \kappa$, $i \in \omega$ such that $k_0\leq i<k_1$ and $g \in \bigcup_{j<l_i}B_{i, j}$, if $\{n \in \omega: \mu \in \supp g(n)(\mu)\}$ then  $\left( \frac{g(n)(\mu)}{K_n}\right)_{n\in A}$ is constant.
			
		\end{enumerate}

	\end{defin}
	
	The family ${\mathcal B}_{i,j}$ is named $(i,j)$-brick. This concept is inspired on the concepts defined in \cite{tomita2015}.
	
	Notice that vi) implies that for every $i<k_1$ $j<l_i$, $h_* \in \mathcal B_{i, j}$, for every $h \in \mathcal B_{i, j}$, $\left(\frac{h(n)(\zeta_i(n))}{h_{*}(n)(\zeta_i(n))}\right)_{n\in A}$ converges to a real number $\theta_{h_*}^h$ and $(\theta_{h_*}^h:\,h \in {\mathcal B}_{i,j})$ is linearly independent.
	
	\medskip

	\section{Constructing a sequence of rational stacks}
	
	\label{section.associated.stack}

	Given a finite sequence of functions we want to find an element of the ultrafilter $p$ that makes the restricted functions closer to  the properties we want for the stack.
	
	\begin{lem} \label{Case A} Suppose that ${\mathcal G}$ is a finite subset $ G^\omega$, $p$ is a selective ultrafilter and $C\in p$. Suppose $\zeta, \zeta_* \in \kappa ^\omega$ are such that there exist $g_* \in {\mathcal G}$ such that $\{ n \in C:\, \zeta(n)\in \supp g_*(n)\} \in p$ and $\{ n\in C:\, \zeta(n)=\zeta_*(n) \} \notin p$.
		
		Then there exist $B'\in p$, $B' \subseteq C$ and $\mathcal H\subseteq \mathcal G$ such that:

		\begin{enumerate}[label=($\star$\arabic*)]
			\item $(\zeta(n))_{n\in B'}$  is either constant or 1-1,
			\item for each $g\in {\mathcal G}$, either $\zeta(n) \in \supp g(n)$ for each $n \in B'$ or  $\zeta(n) \notin \supp g(n)$ for each $n \in B'$,
			\item ${\mathcal H}=\{g \in {\mathcal G}:\forall n \in B'\, \zeta(n) \in \supp g(n)\}$ is nonempty,
			\item $(g(n)(\zeta(n)))_{n\in B'}$ converges strictly monotonically to an element of the extended real line, or it is constant and equal to a rational number, for each $g \in {\mathcal H}$,
			\item  given $f,g \in {\mathcal H}$ either $|g(n)(\zeta(n))> f(n)(\zeta(n))|$ for each $n\in B'$,  $|g(n)(\zeta(n))|= |f(n)(\zeta(n))|$ for each $n\in B'$  or
			$|g(n)(\zeta(n))|< |f(n)(\zeta(n))|$ for each $n\in B'$,
			\item 	for each pair $g, h \in {\mathcal H}$, the sequence $(\frac{g(n)(\zeta(n))}{h(n)(\zeta(n))})_{n\in B'}$ converges to $+\infty$, $-\infty$ or to a real number, and
			\item $\zeta(n) \neq \zeta_*(m) $ for each $n, m \in B'$.
		\end{enumerate} 
	\end{lem}
	
	\begin{proof} Everything follows from the selectivity of $p$.
	\end{proof}
	
	Notice that if $B \in p$ is such that $B\subseteq B'$, then $(\star 1)-(\star 8)$ also hold for $B$.\\

	\begin{lem}		\label{Case A2}
		Suppose that ${\mathcal G}$, $C$, $p$, $\zeta$, $\zeta_*$, $B'$ and $\mathcal H$ be as in Lemma \ref{Case A}.
		
		Suppose $g_\# \in \mathcal H$ is such that for every $g \in \mathcal H$, $(\frac{g(n)(\zeta(n))}{g_{\#}(n)(\zeta(n))})_{n\in B'}$ converges to a real number (or, equivalently, is bounded).
		
		Then there exist $B\in p$, with $B \subseteq B'$, $\mathcal B\subseteq \mathcal H$, $\sigma: \mathcal H\setminus \mathcal B\rightarrow G^\omega$ and a family of real numbers $(\theta_{g_\#, g}: g \in \mathcal H)$ such that:

		\begin{enumerate}[label=($\star$\arabic*)]
			\setcounter{enumi}{8}
			
			\item $g_\# \in \mathcal B$,
			\item $(\frac{g(n)(\zeta(n))}{g_{\#}(n)(\zeta(n))})_{n\in B}$ converges to $\theta_{g_{\#},g}$, for every $g \in {\mathcal H}$,
			\item $(\theta_{g_{\#},g}:\, g \in {\mathcal B})$ is a linearly independent set that generates the same $\mathbb Q$-vector space as $(\theta_{g_\#, g}: g \in \mathcal H)$,
			\item for each $g \in \mathcal H\setminus \mathcal B$, $\mathcal B\cup\{g\}$ and  $\mathcal B\cup\{\sigma(g)\}$ generate the same $\mathbb Q$-vector subspace of $G^\omega$,
			\item for each $g \in \mathcal H\setminus  \mathcal B$ and $h \in \mathcal B$, $\left(\frac{\sigma(g)(n)(\zeta(n))}{h(n)(\zeta(n))}\right)_{n\in B}$ converges to $0$,
			\item If $g \in \mathcal H\setminus \mathcal B$ and $\theta_{g_\#, g}=0$, then $\sigma(g)=g$.
			\item If $\zeta$ is constant and equal to $\nu$, $\mathcal B=\{\chi_{\vec\nu}\}$ and there exists $g^*\in\mathcal H$ such that  $(g^*(n)(\nu))_{n\in B}$ is not constant mod $p$, then $\{n\in\omega:\nu\in\supp\sigma(g^*(n))\}\in p$.
		\end{enumerate} 
		
	\end{lem}
	
	\begin{proof}	
		Consider $\{\theta_{g_{\#},g} :\, g \in {\mathcal H}\}$ as a subset of the ${\mathbb Q}$-vector space ${\mathbb R}$ and take ${\mathcal B}\subseteq {\mathcal H}$ containing $g_{\#}$ such that $(\theta_{g_{\#},h}:\, h \in {\mathcal B})$ is a basis for the vector space generated by $\{\theta_{g_{\#},g} :\, g \in {\mathcal H}\}$.
		
		For the existence of $\sigma$, define $(r_{g, h}: g \in \mathcal H, h \in \mathcal B)$ by the expressions $\theta_{g_{\#}, g}=\sum_{h \in \mathcal B}r_{g,h} \theta_{g_{\#}, h}$. Now define 
		$\sigma(g)= g - \sum_{h \in {\mathcal B}} r_{g,h} h$ for each $g \in \mathcal H\setminus \mathcal B$.
		
	\end{proof}

	\begin{lem} \label{brickline}
		Let $B\in p$. Suppose that $\zeta \in \kappa^\omega$, $m\in\omega$, $\zeta_i\in\kappa^\omega$ for $i<m$, are such that $\{n \in B:\forall i<m\; \zeta(n)\neq\nu_i\} \in p$. Suppose $\mathcal G$ is a finite subset of $G^\omega$ whose elements are distinct mod $p$, none of them is constant mod $p$  and such that $\{[f]_p:f\in \mathcal G\}\cup\{[\chi_{\vec\xi}]_p: \xi \in \kappa \}$ is linearly independent and there exists $g\in\mathcal G$ such that $\{n\in\omega:\zeta(n)\in\supp g(n)\}\in p$.
		
		If $\zeta$ is constant, let $\nu$ be its value.

		Then there exist finite $\mathcal G'\subset G^\omega$, $l\in\omega\setminus\{0\}$, finite non-empty $\mathcal B_j\subset G^\omega$ for each $j<l$, $A\subset B$ such that:
		\begin{enumerate}
			\item For every $g\in\mathcal G'$, $\{n\in\omega:\zeta(n)\in\supp g(n)\}\notin p$ and $\{n\in\omega:\nu_i\in\supp g(n)\}\notin p$ for each $i<m$,
			\item $\mathcal B_j\cap\mathcal B_{j'}=\void$ for $j\ne j'$, $\mathcal B_j\cap\mathcal G'=\void$ for each $j<l$ and $\{[f]_p:f\in\mathcal G'\cup\bigcup_{j<l}\mathcal B_j\}$ is a linearly independent subset of $G^\omega/p$. Also, if $f, h\in\mathcal G'\dot\cup\dot\bigcup_{k<l}\mathcal B_k$ are distinct, then $[f]_p\neq[h]_p$.
			\item As vector subspaces of $G^\omega$, $\langle\mathcal G'\cup\bigcup_{j<l}\mathcal B_j\rangle=\langle\mathcal G\cup\{\chi_{\vec{\nu}}\}\rangle$ if $\zeta$ is constant and $\langle\mathcal G'\cup\bigcup_{j<l}\mathcal B_j\rangle=\langle\mathcal G\rangle$ otherwise,
			\item $\zeta(n) \in \supp h(n)$, for each $j<l$, $h \in {\mathcal B}_{j}$ and $n\in A$,
			\item $\nu_i \notin \supp h(n)$, for each $i<m$, $j< l$ and $h \in {\mathcal B}_{j}$ and $n \in A$,
			\item $(h(n)(\zeta(n)))_{n\in A}$ converges strictly monotonically to an element of the extended real line, or it is constant and equal to a rational number, for each  $j<l$ and $h\in {\mathcal B}_{j}$,
			\item For every $j<l$, there exists $h_* \in \mathcal B_{j}$ such that for every $h \in \mathcal B_{j}$, $\left(\frac{h(n)(\zeta(n))}{h_{*}(n)(\zeta(n))}\right)_{n\in A}$ converges to a real number $\theta_{h_*}^h$ and $(\theta_{h_*}^h:\,h \in {\mathcal B}_{j})$ is linearly independent (as a $\mathbb Q$-vector space),
			\item For each $j'<j<l$, $h \in {\mathcal B}_{j}$ and $h'\in {\mathcal B}_{j'}$ $\left(\frac{h(n)(\zeta(n))}{h'(n)(\zeta(n))}\right)_{n\in A}$ converges monotonically to $0$,
			\item If $\zeta$ is constant, there exists $j<l$ such that $\chi_{\vec{\nu}}\in{\mathcal B}_{j}$,
			\item  For each $j<l$ and distinct $h, h' \in {\mathcal B}_{j}$, either 
			
			\begin{itemize}
				\item $|h(n)(\zeta(n))|> |h'(n)(\zeta(n))|$  for each  $n \in A$, or
				\item   $|h(n)(\zeta(n))|= |h'(n)(\zeta(n))|$  for each  $n \in A$, or
				\item  $|h(n)(\zeta(n))|< |h'(n)(\zeta(n))|$  for each  $n \in A$.
			\end{itemize}
			\item No element of $\mathcal G'$ is constant mod $p$ and $\{[f]_p:f\in \mathcal G'\}\cup\{[\chi_{\vec\xi}]_p: \xi \in \kappa \}$ is linearly independent,
			\item if $i<m$ and $n, n' \in A$ are distinct, then $\zeta(n)\neq \zeta_i(n')$, and
			\item $|\mathcal G'|<|\mathcal G|$.
		\end{enumerate}

		\begin{proof}
			
			Since $p$ is selective, we may suppose by shrinking $B$ is necessary that if $i<m$ and $n, n' \in B$ are distinct, then $\zeta(n)\neq \zeta_i(n')$. Clearly, this property will hold for any subset of $B$.
			
			If $\zeta$ is constant, let $\mathcal G_{0}=\mathcal G\cup\{\chi_{\vec\nu}\}$. If not, let $\mathcal G_0=\mathcal G$. We will construct:

			\begin{itemize}
				\item $A_{j} \in p$ for $j\in \omega$, with $A_j\subseteq B$,
				\item $\mathcal G_{j}\subseteq G^\omega$ for $j\in \omega$,
				\item $\mathcal H_{j}\subseteq{\mathcal G}_{j}$ for $j\in \omega$,
				\item ${\mathcal B}_{j}\subseteq\mathcal H_{j}$ for $j\in \omega$,
				\item $\sigma_{j}:\mathcal H_{j}\setminus{\mathcal B}_{j}\rightarrow G^\omega$ for $j\in \omega$,
			\end{itemize} 
			satisfying:	
			\begin{enumerate}[label=\roman*)]
				\item $A_{j}\subseteq A_{j-1}$ for every $j<l$,
				\item for each $j<l$, $\{[f]_p:f\in\mathcal G_j\dot\cup\dot\bigcup_{k<j}\mathcal B_k\}$ is a linearly independent subset of $G^\omega/p$. Also, if $f, h\in\mathcal G_j\dot\cup\dot\bigcup_{k<j}\mathcal B_k$ are distinct, then $[f]_p\neq[h]_p$.
				
				\item if $j\in\omega$, $h \in {\mathcal H}_{j}$ and $n\in A_j$, then $\zeta(n) \in \supp h(n)$,
				\item if $j\in\omega$, $h \in {\mathcal G}_j\setminus {\mathcal H}_{j}$ and $n\in A_j$, then $\zeta(n) \notin \supp h(n)$,
				
				\item if $j\in\omega$, $h \in {\mathcal H}_{j}$, $i<m$ and $n\in A_j$, then $\nu_i \notin \supp h(n)$,
				\item if $j\in\omega$ and $h \in {\mathcal H}_{j}$, then  $(h(n)(\zeta(n)))_{n\in A_j}$ converges strictly monotonically to an element of the extended real line, or it is constant and equal to a rational number,
				\item for every $j\in\omega$, $\mathcal B_j\neq \void$ iff $\exists g \in \mathcal G_j\,\{n\in \omega: \zeta(n) \in \supp g(n)\}\in p$,
				\item for every $j<\omega$, if $\mathcal B_j\neq \void$, then there exists $h_* \in \mathcal B_{j}$ such that for every $h \in \mathcal B_{j}$, $\left(\frac{h(n)(\zeta(n))}{h_{*}(n)(\zeta(n))}\right)_{n\in A_j}$ converges to a real number $\theta_{h_*}^h$ and $(\theta_{h_*}^h:\,h \in {\mathcal B}_{j})$ is linearly independent (as a $\mathbb Q$-vector space),
				\item for every $j, j' \in \omega$, if $j'<j$, $h \in {\mathcal B}_{j}$ and $h'\in {\mathcal B}_{j'}$, then $\left(\frac{h(n)(\zeta(n))}{h'(n)(\zeta(n))}\right)_{n\in A_j}$ converges monotonically to $0$,
				\item if $\zeta$ is constant, there exists $j\in \omega$ such that $\chi_{\vec{\nu}}\in{\mathcal B}_{j}$,
				\item if $j\in \omega$ and $h, h' \in {\mathcal B}_{j}$, either 
				
				\begin{itemize}
					\item $|h(n)(\zeta(n))|> |h'(n)(\zeta(n))|$  for each  $n \in A_j$, or
					\item   $|h(n)(\zeta(n))|= |h'(n)(\zeta(n))|$  for each  $n \in A_j$, or
					\item  $|h(n)(\zeta(n))|< |h'(n)(\zeta(n))|$  for each  $n \in A_j$.
				\end{itemize}

				\item for each $j\in \omega$ and $g \in \mathcal H_j$, $(|g(n)(\zeta(n))|)_{n\in A_j}$ is either constant or strictly increasing,
				\item for each $j\in \omega$ and $g \in \mathcal H_j\setminus \mathcal B_j$, $\mathcal B_j\cup\{g\}$ and  $\mathcal B_j\cup\{\sigma(g)\}$ generate the same $\mathbb Q$-vector subspace of $G^\omega$,
				\item for each $j\in \omega$, $g \in \mathcal G_j$ and $i<m$, $\{n\in\omega:\nu_i\in\supp g(n)\}\notin p$,
				\item for each $j, j'\in \omega$, $\mathcal G_j\cup\bigcup_{k< j}\mathcal B_k$ generates the same subspace of $G^\omega$ as $\mathcal G_{j'}\cup\bigcup_{k< j'}\mathcal B_k$\:
				\item if $\zeta$ is constant, then for each $j\in \omega$, $\chi_{\vec\nu} \in \bigcup_{k< j}\mathcal B_j\cup \mathcal G_j$,
				\item $\mathcal G_{j+1}=(\mathcal G_j\setminus \mathcal H_j) \cup \ran \sigma_j$, and
				\item if $\zeta$ is constant, $j \in \omega$ and $\chi_{\vec \nu} \in \mathcal H_j\setminus \mathcal B_j$, then $\sigma_j(\chi_{\vec \nu})=\chi_{\vec \nu}$.
				\item if $\zeta$ is constant and $\mathcal B_0=\{\chi_{\vec \nu}\}$, then there exists $g \in \mathcal H_0\setminus \mathcal B_0$ such that $\{n \in \omega: \nu \in \supp \sigma_0(g)\} \in p$.

			\end{enumerate}
			
			Suppose we have carried on such a recursion. By ii) and xviii), one of the $\mathcal B_j$'s must be empty. Let $l$ be the first $j$ such that $\mathcal B_j=\void$. By vii), $\forall g \in \mathcal G_l,\{n \in \omega: \zeta(n) \notin g(n)\}\in p$. Since there exists $g\in\mathcal G$ such that $\{n\in\omega:\zeta(n)\in\supp g(n)\}\in p$, it follows that $j>0$. Let $A=A_{l-1}$, $\mathcal G'=\mathcal G_{l}$. Notice that every $\mathcal B_j$ is nonempty for $j<l$.
			
			(1) holds by the previous observation, i), v) and by the fact that $\mathcal B_j\subseteq\mathcal H_j$. (2) holds by (ii). (3) follows from xvii) using $j=l$, $j'=0$. (4)-(8), (10) and (11) follow easily from (i), (iii)-(vii), (viii), (ix), (xi) and (xii). Suppose $(9)$ doesn't hold. Then by xviii), $\chi_{\vec\nu} \in \mathcal G_l$. But then, by vii), $\mathcal B_l\neq \void$, a contradiction.
			
			(13) holds by ii), because $\langle\mathcal G'\rangle\subseteq \langle \mathcal G_0\rangle$ and because, if $\zeta$ is constant then, by xviii) and (9), $\mathcal G'\cup \{\chi_{\vec \nu}\}$ is linearly independent.
			
			(14) holds: if $\zeta$ is not constant, it follows from (2) and (3). If it is constant, first, notice that, by xxi), xix) for $j=0$, and vii) for $j=1$, it follows that $l>1$ or $\mathcal B_0\neq \{\chi_{\vec \nu}\}$. Either way, $\mathcal B=\bigcup_{i<l}\mathcal B_i\setminus \{\chi_{\vec \nu}\}$ is nonempty. By (2), (3) and (9), analyzing dimensions it follows that $1+|\mathcal B|+|\mathcal G'|=1+|\mathcal G|$, so $|\mathcal G'|<|\mathcal G|$.
			
			\textbf{Construction:} For step $0$, $\mathcal G_0$ is already defined. We apply Lemma \ref{Case A} $m$ times using $\zeta(n)=\zeta(n)$ and $\zeta_*(n)=\nu_i$ for every $n$. If $m=0$ we apply it once using $\zeta_*(n)=\zeta(n)'$ for every $n$ for some $\zeta(n)'\neq \zeta(n)$. We now have $\mathcal H_0$ and $A'_0\subset B$.
			
			If it is the case that $(h(n)(\zeta(n)))_{n\in A_0}$ converges to a real number for every $h\in\mathcal H_0$ and that $\zeta$ is constant, then we apply Lemma \ref{Case A2} with $g_\#=\chi_{\vec\nu}$, and obtain $A_0\subset A'_0$, $\mathcal B_0\subset\mathcal H_0$ and $\sigma_0$. If not, then we take any $g_\#\in\mathcal H_0$ that satisfies the hypothesis of Lemma \ref{Case A2} -- one does exist because of ($\star5$), which also implies that for such a $g_\#$, $\left(\frac{1}{g_\#(n)(\zeta(n))}\right)_{n\in A_0}$ converges to 0, and thus in case $\zeta$ is constant,  $\sigma_0(\chi_{\vec\nu})=\chi_{\vec\nu}$. Either way, we obtain $\sigma_0$, $\mathcal B_0$ and $\mathcal A_0$.  It is straightforward to verify that $(i)-(xxi)$ hold for this step. 
			
			For the inductive step, we define $\mathcal G_{j+1}$ as in xix). If $\nexists g \in \mathcal G_{j+1}\,\{n\in \omega: \zeta(n) \in \supp g(n)\}\in p$, then we define $\mathcal H_{j+1}=\void$, $A_{j+1}\subseteq A_j$ satisfying v) with $j$ swapped by $j+1$ and $\mathcal B_{j+1}=\sigma_{j+1}=\void$. If not, we proceed as in step $0$: we first apply Lemma \ref{Case A} to obtain $\mathcal H_{j+1}$ and $A'_{j+1}\subset A_j$ and then similarly apply Lemma \ref{Case A2} to obtain $\mathcal B_{j+1}$, $A_{j+1}$ and $\sigma_{j+1}$.	It is straightforward to verify that (i)-(xix) hold for this step.

		\end{proof}

	\end{lem}

	\begin{lem} \label{aux lem}
		Suppose $\mathcal G$ is a finite subset of $G^\omega$ whose elements are distinct mod $p$ and none of them is constant mod $p$ such that $\{[f]_p:f\in \mathcal G\}\cup\{[\chi_{\vec\nu}]_p:\nu\in\kappa\}$ is a linearly independent subset of $G^\omega/p$. Then: either there exist $\mu\in\kappa$, $g^*\in\mathcal G$ and $A\in p$ such that $(g^*(n)(\mu))_{n\in A}$ is one-to-one, or there exists $A\in p$ such that for every $g\in\mathcal G$ there exists $\zeta_g\in\kappa^\omega$ satisfying $\zeta_g(n)\in\supp g(n)$ for all $n\in A$ and $\zeta_g|A$ is one-to-one.
	\end{lem}
	
	\begin{proof}
		Suppose that for all $\mu\in\kappa$, for all $g\in\mathcal G$ and for all $A\in p$, $(g(n)(\mu))_{n\in A}$ is not one-to-one. Then by the selectivity of $p$, for all $\mu\in\kappa$ there exist $B_\mu\in p$ such that for all $g\in\mathcal G$, $(g(n)(\mu))_{n\in B_\mu}$ is constant.\\
		Fix a $g\in\mathcal G$. By selectivity, there exists $B\in p$ such that either the sequence $(|\supp g(n)|)_{n\in B}$ is strictly increasing or it is constant. If it is strictly increasing, then we may pick recursively $\tilde\zeta_g(n)\in\supp g(n)$ for each $n\in B$ in a way that $\tilde\zeta_g$ is one-to-one. Define $A_g=B$.\\
		If it is constant, let $k\in\omega$ be that constant. Since $g$ is not 0 mod $p$, $k\ge1$. For each $i<k$, let $w_i\in\kappa^\omega$ be such that $\supp g(n)=\{w_0(n),\ldots,w_{k-1}(n)\}$ for each  $n\in B$. Then there exists $C\subset B$, $C\in p$ such that for each $i<k$, $(w_i(n))_{n\in C}$ is one-to-one or constant.\\
		We claim that there is a $j<k$ such that $(w_j(n))_{n\in C}$ is one-to-one. Suppose all of them are constant; take $\mu_j$ for each $j<k$ so that $w_j(n)=\mu_j$ for all $n\in C$. Then, since $(g(n)(\mu_j))_{n\in B_{\mu_j}}$ is constant for each $j<k$, let $r_j$ be those constants. Let $D=C\cap\bigcap_{j<k}B_{\mu_j}$. We have that $D\in p$ and $g(n)=\left(\sum_{j<k}r_j\chi_{\overrightarrow{\mu_j}}\right)(n)$ for all $n\in D$, and so $[g]_p=\sum_{j<k}r_j[\chi_{\overrightarrow{\mu_j}}]_p$. This contradicts the hypothesis that $\{[f]_p:f\in \mathcal G\}\cup\{[\chi_{\vec\nu}]_p:\nu\in\kappa\}$ is a linearly independent subset of $G^\omega/p$.\\
		Therefore, there exists a $j<k$ such that  $(w_j(n))_{n\in C}$ is one-to-one, and so we define $\tilde\zeta_g=w_j$ and $A_g=C$.\\
		Thus if we define $A=\bigcap_{g\in\mathcal G}A_g$ and $\zeta_g\in\kappa^\omega$ such that $\zeta_g|A=\tilde\zeta_g|A$, we have the desired result.
	\end{proof}

	Now we restate Lemma \ref{Lem.stack.1}, which we are going to prove.
	
	\begin{lemnumber} 
		Let $B \in p$ and $\mathcal G$ be a finite subset of $G^\omega$ whose elements are distinct mod $p$ and none of them is constant mod $p$ such that $\{[f]_p:f\in \mathcal G\}\cup\{[\chi_{\vec\nu}]_p:\nu\in\kappa\}$ is linearly independent. Then
		there exist a rational stack  $\mathcal S=\langle  \mathcal B, \nu, \zeta, K, A, k_0, k_1, l\rangle$ and a positive integer $S$ such that, by defining
		${\mathcal A}={\mathcal G} \cup \{ \chi_{\vec{\nu_i}}:\, i<k_0\}$ and ${\mathcal C}= \frac{\bigcup_{i< k_1, j<l_i} {\mathcal B}_{i,j}}{{K}}$, there exist  ${\mathcal M}:\mathcal A\times \mathcal C\rightarrow \mathbb Z$, ${\mathcal N}:\mathcal C\times \mathcal A\rightarrow \mathbb Z$ satisfying:
		
		\begin{enumerate}[label=(\arabic*)]
			\item $\{ [f]_p:\, f \in {\mathcal A} \}$ and $\{ [h]_p:\, h \in {\mathcal C} \}$ generate the same subspace;
			\item $f(n) =\sum_{h \in {\mathcal C}}{\mathcal M}_{f,h}h(n)$, for each $n\in A$ and $f \in {\mathcal A}$,
			\item $h(n) = \frac{1}{S}.\sum_{f\in{\mathcal A}}{\mathcal N}_{h,f}f(n)$, for each $n\in A$ and $h\in{\mathcal C}$,
			
			\item $K\mathcal A\subseteq H^\omega$,
			\item $K\mathcal C\subseteq H^\omega$, and
			\item $A\subseteq B$.
		\end{enumerate}
	\end{lemnumber}

	\begin{proof} (of Lemma \ref{Lem.stack.1}) 
		We will start building a sequence that will almost be the stack $\mathcal S$ which we will associate with  $\mathcal G$.

		\textbf{Claim:} There exist:
		
		\begin{itemize}
			\item $A'\in p$ with $A'\subseteq B$,
			\item $k_0\in \omega$
			\item $l':k_0\rightarrow \omega\setminus \{0\}$,
			\item $\nu:k_0 \rightarrow \kappa$,
			\item $\zeta':k_0\rightarrow \kappa^\omega$,
			\item $\mathcal G'\subset G^\omega$,
			\item $(\hat{\mathcal B}_{i, j}: i<k_0, j<l_i)$ a family of nonempty subsets of $G^\omega$,
			
		\end{itemize} 
		Satisfying:
		
		\begin{enumerate}
			\item $\zeta'_i(n)=\nu_i$ for every $i<k_0$ and $n \in A'$,
			\item  The elements $\nu_i\, (i<k_0)$ are pairwise distinct,
			\item $\nu_i \in \supp h(n)$, for each $ i<k_0$, $j<l_i$, $h \in \hat{\mathcal B}_{i,j}$ and $n\in A$,
			\item $\nu_{i} \notin \supp h(n)$, for each, $i<i_*<k_0$, $j< l_{i_*}$ and $h \in \hat{\mathcal B}_{i_*,j}$ and $n \in A$,
			\item $(h(n)(\nu_i))_{n\in A}$ converges strictly monotonically to an element of the extended real line, or it is constant and equal to a rational number, for each $i< k_0$, $j<l_i$ and $h\in\hat{\mathcal B}_{i,j}$,
			\item For every $i<k_0, j<l_i$, there exists $h_* \in \hat{\mathcal B}_{i, j}$ such that for every $h \in \hat{\mathcal B}_{i, j}$, $\left(\frac{h(n)(\nu_i)}{h_{*}(n)(\nu_i)}\right)_{n\in A}$ converges to a real number $\theta_{h_*}^h$ and $(\theta_{h_*}^h:\,h \in\hat{\mathcal B}_{i,j})$ is linearly independent (as a $\mathbb Q$-vector space),
			\item For each $i<k_0$, $j'<j<l_i$, $h \in\hat{\mathcal B}_{i,j}$ and $h'\in \hat{\mathcal B}_{i,j'}$, $\left(\frac{h(n)(\nu_i)}{h'(n)(\nu_i)}\right)_{n\in A}$ converges monotonically to $0$,
			\item $\left(|h(n)(\nu_i)|\right)_{n\in A'}$ is constant or strictly increasing, for each $i<k_0, j<l_i$ and $h \in \hat{\mathcal B}_{i,j}$,
			\item For each $ i<k_0$ there exists $j<l_i$ such that $\chi_{\vec{\nu_i}}\in\hat{\mathcal B}_{i,j}$,
			\item  For each $i<k_0, j<l_i$ and distinct $h, h_* \in\hat{\mathcal B}_{i,j}$, either 
			
			\begin{itemize}
				\item $|h(n)(\nu_i)|> |h_*(n)(\nu_i)|$  for each  $n \in A$, or
				\item   $|h(n)(\nu_i)|= |h_*(n)(\nu_i)|$  for each  $n \in A$, or
				\item  $|h(n)(\nu_i)|< |h_*(n)(\nu_i)|$  for each  $n \in A$.
			\end{itemize}
			
			\item for all $\mu \in \kappa$, for every $g \in \mathcal G'$, if $\{n \in \omega: \mu \in \supp g(n)(\mu)\}\in p$ then  $(g(n)(\mu))_{n\in A}$ is constant,
			\item for all $g\in\mathcal G'$ and all $i<k_0$, $\{n\in\omega:\nu_i\in\supp g(n)\}\notin p$,
			\item If $i, i'<k_0$, $j<l'_i$, $j'<l'_{i'}$ and $(i, j)\neq (i', j')$, then $\hat{\mathcal B}_{i,j}\cap\hat{\mathcal B}_{i', j'}=\void$, $\hat{\mathcal B}_{i, j}\cap\mathcal G'=\void$ for each $j<l$ and $\{[f]_p:f\in\mathcal G'\cup\bigcup_{i<k_0}\bigcup_{j<l_i}\hat{\mathcal B}_{i, j}\}$ is a linearly independent subset of $G^\omega/p$. Also, if $f, h\in\mathcal G'\dot\cup\dot\bigcup_{i<k_0}\dot\bigcup_{j<l_i}\hat{\mathcal B}_{i, j}$ are distinct, then $[f]_p\neq[h]_p$.
			\item As vector subspaces of $G^\omega$, $\langle\mathcal G'\cup\bigcup_{i<k_0}\bigcup_{j<l_i}\hat{\mathcal B}_{i, j}\rangle=\langle\mathcal G\cup\{\chi_{\vec{\nu_i}}:i<k_0\}\rangle$
			\item No element of $\mathcal G'$ is constant mod $p$ and $\{[f]_p:f\in \mathcal G'\}\cup\{[\chi_{\vec\xi}]_p: \xi \in \kappa \}$ is linearly independent.
			
		\end{enumerate}
		
		If $(11)$ already holds for $\mathcal G$ and $A=B$, we let $\mathcal G'=\mathcal G$, $A'=B$, $k_0=0$ and the other sequences be $\void$.
		
		If not, then we may take a $\nu_0\in\kappa$ such that there exist $g\in\mathcal G$ and $B'\subset B$, $B'\in p$ such that $\{n\in\omega:\nu_0\in\supp g(n)\}\in p$ and $(g(n)(\nu_0))_{n\in B'}$ is one-to-one.
		
		We define $\mathcal G_0=\mathcal G$ and apply Lemma \ref{brickline} to $B'$, $m=0$, $\mathcal G_0$, $\zeta(n)=\nu_0$ for all $n\in B'$ and obtain $\mathcal G'_0$, $A_0$, $l_0$ and $\hat{\mathcal B}_{0,j}$ for $j<l_0$ satisfying everything but (11) (possibly) by using $k_0=0$ .
		
		Suppose the recursion has been done up to $m\in\omega$ and we have, for $i<m$, $\nu_i$ $\mathcal G'_i$, $A_i$, $l_i$ and $\hat{\mathcal B}_{i,j}$ for $j<l_i$ satisfying everything but (10) (possibly) by using $k_0=m$. If (10) holds for $\mathcal G'_{m-1}$, then we let $A'=A_{m-1}$, $k_0=m$ and $\mathcal G'=\mathcal G'_{m-1}$ and the recursion is over. If not, we take $\nu_m\in\kappa$ such that  there exist $g\in\mathcal G'_{m-1}$ and $B'\subset A_{m-1}$, $B'\in p$ such that $\{n\in\omega:\nu_m\in\supp g(n)\}\in p$ and $(g(n)(\nu_m))_{n\in B'}$ is one-to-one. Notice that item (11) implies that $\nu_m\neq\nu_i$ for every $i<m$. We then apply Lemma \ref{brickline} to  $B'$, $m$, $\mathcal G'_{m-1}$, $\zeta(n)=\nu_m$ for all $n\in B'$, $\zeta_i(n)=\nu_i$ for all $n\in B'$ and $i<m$, and obtain $\mathcal G'_m$, $A_m$, $l_m$ and $\hat{\mathcal B}_{m,j}$ for $j<l_m$ satisfying everything but (11) (possibly) by using $k_0=m$.\\
		The recursion must eventually stop due to items (12), (13) and the fact that the $\hat{\mathcal B}_{i,j}$'s are nonempty. We now have $A'$, $k_0$, $l'$, $\nu$, $\zeta'$, $\mathcal G'$ and $(\hat{\mathcal B}_{i, j}: i<k_0, j<l_i)$ as in the Claim above.\\
		\textbf{Claim:}  There exist:
		
		\begin{itemize}
			\item $A''\in p$ with $A''\subseteq A'$,
			\item $k_1\in \omega\bs\{0\}$
			\item $l:k_1\rightarrow \omega\setminus \{0\}$ extending $l'$,
			\item $\zeta:k_1\rightarrow \kappa^\omega$ extending $\zeta'$,
			\item $(\hat{\mathcal B}_{i, j}: i<k_1, j<l_i)$ a family of nonempty subsets of $G^\omega$,
			
		\end{itemize} 
		Satisfying:
		
		\begin{enumerate}
			\item $\zeta_i(n)=\nu_i$ for every $i<k_1$ and $n \in A''$,
			\item  The elements $\nu_i\, (i<k_0)$ and $\zeta_j(n)$ $(k_0\le j<k_1, n\in A)$ are pairwise distinct,
			\item $\zeta_i(n) \in \supp h(n)$, for each $ i<k_1$, $j<l_i$, $h \in \hat{\mathcal B}_{i,j}$ and $n\in A''$,
			\item $\zeta_i(n) \notin \supp h(n)$, for each, $i<i_*<k_1$, $j< l_{i_*}$ and $h \in \hat{\mathcal B}_{i_*,j}$ and $n \in A''$,
			\item $(h(n)(\zeta_i(n)))_{n\in A''}$ converges strictly monotonically to an element of the extended real line, or it is constant and equal to a rational number, for each $i< k_1$, $j<l_i$ and $h\in\hat{\mathcal B}_{i,j}$,
			\item For every $i<k_1, j<l_i$, there exists $h_* \in \hat{\mathcal B}_{i, j}$ such that for every $h \in \hat{\mathcal B}_{i, j}$, $\left(\frac{h(n)(\zeta_i(n))}{h_{*}(n)(\zeta_i(n))}\right)_{n\in A''}$ converges to a real number $\theta_{h_*}^h$ and $(\theta_{h_*}^h:\,h \in\hat{\mathcal B}_{i,j})$ is linearly independent (as a $\mathbb Q$-vector space),
			\item For each $i<k_1$, $j'<j<l_i$, $h \in\hat{\mathcal B}_{i,j}$ and $h'\in \hat{\mathcal B}_{i,j'}$, $\left(\frac{h(n)(\zeta_i(n))}{h'(n)(\zeta_i(n))}\right)_{n\in A''}$ converges monotonically to $0$,
			\item $\left(|h(n)(\zeta_i(n))|\right)_{n\in A''}$ is strictly increasing, for each $i<k_1, j<l_i$ and $h \in \hat{\mathcal B}_{i,j}$,
			\item For each $ i<k_0$ there exists $j<l_i$ such that $\chi_{\vec{\nu_i}}\in\hat{\mathcal B}_{i,j}$,
			\item  For each $i<k_1, j<l_i$ and distinct $h, h_* \in\hat{\mathcal B}_{i,j}$, either 
			
			\begin{itemize}
				\item $|h(n)(\zeta_i(n))|> |h_*(n)(\zeta_i(n))|$  for each  $n \in A''$, or
				\item   $|h(n)(\zeta_i(n))|= |h_*(n)(\zeta_i(n))|$  for each  $n \in A''$, or
				\item  $|h(n)(\zeta_i(n))|< |h_*(n)(\zeta_i(n))|$  for each  $n \in A''$.
			\end{itemize}
			
			\item for all $\mu \in \kappa$, for every $i\ge k_0$, $j<l_i$ and $g \in \hat{\mathcal B}_{i,j}$, if $\{n \in \omega: \mu \in \supp g(n)(\mu)\}\in p$ then  $(g(n)(\mu))_{n\in A''}$ is constant,
			\item As vector subspaces of $G^\omega$, $\langle\bigcup_{i<k_0}\bigcup_{j<l_i}\hat{\mathcal B}_{i, j}\rangle=\langle\mathcal G\cup\{\chi_{\vec{\nu_i}}:i<k_0\}\rangle$,
			\item If $i, i'<k_1$, $j<l'_i$, $j'<l'_{i'}$ and $(i, j)\neq (i', j')$, then $\hat{\mathcal B}_{i,j}\cap\hat{\mathcal B}_{i', j'}=\void$,  and $\{[f]_p:f\in\bigcup_{i<k_0}\bigcup_{j<l_i}\hat{\mathcal B}_{i, j}\}$ is a linearly independent subset of $G^\omega/p$. Also, if $f, h\in\bigcup_{i<k_0}\bigcup_{j<l_i}\hat{\mathcal B}_{i, j}$ are distinct, then $[f]_p\neq[h]_p$.\\
			
				\end{enumerate}
			For the initial step $k_0$ of the recursion, we first notice that by item (11) of the previous Claim, applying Lemma \ref{aux lem} we may take $A''_{k_0}\in p$ such that for every $g\in\mathcal G'$ there exists $\zeta_g\in\kappa^\omega$ such that $\zeta_g(n)\in\supp g(n)$ for all $n\in A''_{k_0}$ and $\zeta_g|A''_{k_0}$ is one-to-one. Define $\mathcal G_{k_0}=\mathcal G'$, and apply Lemma \ref{brickline} to  $B=A''_{k_0}$, $m=k_0$,  $\mathcal G_{k_0}$, $\zeta=\zeta_g$ for an arbitrary $g\in\mathcal G'$, and obtain  $\mathcal G'_{k_0}$, $A''_{k_0+1}$, $l_{k_0+1}$ and $\hat{\mathcal B}_{k_0+1,j}$ for $j<l_{k_0+1}$. We then repeat this step until for some $k_1\ge k_0$, $\mathcal G'_{k_1}=\void$. Such a $k_1$ exists since, by Lemma \ref{brickline}, $|\mathcal G'_m|<|\mathcal G'_{m-1}|$ for each $m>k_0$.
			
		It follows from items (12) and (13) that $\bigcup_{i<k_0}\bigcup_{j<l_i}\hat{\mathcal B}_{i, j}$ and $\mathcal G\cup\{\chi_{\vec{\nu_i}}:i<k_0\}$, mod $p$, are bases for the same subspace of $G^\omega/p$. Let $\mathcal A=\mathcal G\cup\{\chi_{\vec{\nu_i}}:i<k_0\}$ and $\mathcal B'=\bigcup_{i<k_0}\bigcup_{j<l_i}\hat{\mathcal B}_{i, j}$.
		
		Fix families of integers $\hat{\mathcal M}=(\hat{\mathcal M}_{f, h}: f \in \mathcal A, g \in \mathcal B')$ and $\hat{\mathcal N}=(\hat{\mathcal N}_{h, f}: h \in \mathcal B', f \in \mathcal A)$ and a positive integer $T$ such that:
		
		$(3')$ $[f]_p = \frac{1}{T}\sum_{h \in {\mathcal B}'}\hat{\mathcal M}_{f,h}[h]_p$, for each $f \in {\mathcal A}$ and
		
		$(4')$ $[h]_p = \frac{1}{T}\sum_{f\in{\mathcal A}}\hat{\mathcal N}_{h,f}[f]_p$, for each $h\in{\mathcal B}'$.
		
		Let ${\mathcal C}= \{ \frac{h}{T}:\, h \in {\mathcal B}'\}$ and ${\mathcal M}$, ${\mathcal N}$ be such that ${\mathcal M}_{f,\frac{h}{T}}=\hat{\mathcal M}_{f,h}$ and ${\mathcal N}_{\frac{h}{T},f}=\hat{\mathcal N}_{h,f}$. Then:

		$(3'')$ $[f]_p = \sum_{h \in {\mathcal B}}{\mathcal M}_{f,h}[h]_p$, for each $f \in {\mathcal A}$ and
		
		$(4'')$ $[h]_p = \frac{1}{T^2}\sum_{f\in{\mathcal A}}{\mathcal N}_{h,f}[f]_p$, for each $h\in{\mathcal C}$.
		
		Let $A\subset A''$, $A\in p$ be such that for every $n\in A$,  $f \in {\mathcal A}$ and $h\in{\mathcal C}$, $f(n) =\sum_{h \in {\mathcal C}}{\mathcal M}_{f,h}h(n)$ and $h(n) = \frac{1}{T^2}.\sum_{f\in{\mathcal A}}{\mathcal N}_{h,f}f(n)$.
		
		Now let $K$ be a strictly increasing sequence of positive integers such that $K_0>1$, $n!T\mid K_n$ for all $n\in\omega$, and $K.\mathcal C\subset H^\omega$. We now have that by defining $\mathcal B_{i,j}=K.\frac{\hat{\mathcal B}_{i,j}}{T}$, we have the desired rational stack.
	\end{proof}

	\section{Solving arc functions on a level of a rational stack}
	\label{section.stack.arcequation}
	
	The main tool to solve the arc equation of rational stacks is the same used for integer stacks and we state the lemmas used in \cite{tomita2015}. However, there is a crucial difference in the way the stack was defined so that we separate when the denominator grows compared to the numerator, when the numerator and denominator are pretty much at even speed and when the numerator grows compared to the denominator.
	
	\subsection{An application of Kronecker's Lemma} \label{subsection.kronecker}
	
	Kronecker's Lemma says that if $\{1,  \theta_0, \ldots , \theta_{k-1}\}$ is a linearly independent subset of the vector space ${\mathbb R}$ over the
	field ${\mathbb Q}$ then $\{ (\theta_0 . n +{\mathbb Z}, \ldots , \theta_{k-1}.n +{\mathbb Z} ):\, n \in {\mathbb Z}\}$ is a dense subset of ${\mathbb T}^k$ (see \cite{brocker&tomdieck}).

	\begin{lem} ( Lemma 4.3 \cite{tomita2015}) \label{dense.torus} If  $( \theta_0, \ldots , \theta_{r-1})$ is linearly independent subset of the $\mathbb Q$-vector space ${\mathbb R}$ and $\epsilon >0$ then there exists a positive integer $L$ such that $\{ (\theta_0 . x +{\mathbb Z}, \ldots ,\theta_{r-1}.x +{\mathbb Z} ):\, x \in I\}$ is $\epsilon$-dense in the usual Euclidean metric topology, for any interval $I$ of length at least $L$.
	\end{lem}

	Given $\epsilon>0$ and $\theta=(\theta_i:i\in I)$ a finite linearly independent subset of $\mathbb R$ as a $\mathbb Q$-vector space, fix an integer $L(\theta,\epsilon)$ satisfying the conditions in Lemma \ref{dense.torus}.
	
	\begin{lem} (Lemma 4.4 of \cite{tomita2015}) \label{lem.kronecker} Fix a positive real $\epsilon^*<\frac{1}{8}$. Let $\theta=( \theta_0, \ldots , \theta_{r-1})$ be a linearly independent subset of ${\mathbb R}$ as a ${\mathbb Q}$-vector space.
		
		Set $L=L(\theta,\epsilon^*)$  and $(a_0, \ldots , a_{r-1})$ be a sequence of integers such that
		
		$i)$ $|a_0|>\ldots >|a_{r-1}|$ and
		
		$ii)$ $|\theta_k -\frac{a_k}{a_0}| < \frac{\epsilon^*}{\sqrt{r}.L}$ for each $k<r$.
		
		Then
		
		$a)$ $\{(a_0.x , \ldots , a_{r-1}.x):\, x \in J\}$ is $2.\epsilon^*$-dense for any arc $J$
		of length at least $\frac{L}{|a_0|}$ and
		
		$b)$ for any arc $J$ of length at least $3.\frac{L}{|a_0|}$ and ${\mathcal U}$ any open ball of radius $4.\epsilon^*$ (in $\mathbb T^r$ with the Euclidean metric), there exists an arc $K$ contained in $J$ of length $\frac{4.\epsilon^*}{\sqrt{r}.|a_0|}$ such that $\{(a_0.x , \ldots , a_{r-1}.x):\, x \in K\}\subseteq {\mathcal U}$.
	\end{lem}
	
	\subsection{Solving arc equations on a level of a rational stack}
	
	We will present the proof of the remaining technical lemma.
	
	\medskip
	
	\begin{lemn} [\ref{Lem.stack.2}] 
		
 Let $\mathcal S$, ${\mathcal A}$, ${\mathcal C}$, ${\mathcal M}$ and ${\mathcal N}$ be as in Lemma  \ref{Lem.stack.1}. Let $\epsilon$ be a positive real and $D$ be a finite subset of $\kappa$. Then there exist $B \subseteq A$ cofinite in $A$ and a family of positive real numbers $(\gamma_n: n \in B)$ such that:
		
		For every $n \in B$, for every family  $( W_h:\, h \in {\mathcal C} )$ of open arcs of length $\epsilon$, and for every arc function $\psi$ of length $\epsilon$ such that $\supp \psi \subseteq D\setminus \{\nu_0,\ldots \nu_{k_0-1}\}$, there exists an $n$-solution of length $\gamma_n$ for the arc equation $(\psi,B,K.\mathcal C,K_n,W)$.

	\end{lemn}
	
	\begin{proof} For each $(i, j)$ with $i<k_1$ and $j<l_i$, let $u_{i, j} \in \mathcal B_{i, j}$ and $v_{i, j}$ be such that for every $h \in \mathcal B_{i, j}$ and $n \in A$, $|u_{i, j}(n)(\zeta_i(n))|\leq |h(n)(\zeta_i(n))|\leq |v_{i, j}(n)(\zeta_i(n))|$. Fix $ \epsilon^*<\min \{\frac{1}{8} , \frac{1}{8}\epsilon\}$. For each $i<k_1$, $j<l_i$ and $h \in \mathcal B_{i, j}$, let $\theta_h$ be the limit of $\left(\frac{h(n)(\zeta_i(n))}{v_{i, j}(n)(\zeta_i(n))}\right)_{n \in A}$. Let $\theta_{i,j}=(\theta_{h}:h\in\mathcal B_{i,j})$ and let $L$ be a fixed integer greater than $L(\theta_{i,j},\epsilon^*)$ for any $i<k_1$ and $j<l_i$.

		Let $B\subset A$ be the set of $n$'s in $A$ is such that

		$a)$ $\left|\theta_{h} -\frac{h(n)(\zeta_i(n))}{v_{i, j}(n)(\zeta_i(n))}\right| < \frac{\epsilon^*}{\sqrt{|{\mathcal A}|+1}.L}$, for each $i< k_1$, $j<l_i$  and $h \in \mathcal B_{i, j}$;
		
		$b)$ $3.\frac{L}{|v_{i,l_{i}-1}(n)(\zeta_i(n))|}< \epsilon^*$, for
		each $i<k_1$.
		
		$c)$ $\frac{3L}{|u_{i,j-1}(n)(\zeta_i(n))|} \leq\frac{4.\epsilon^*}{\sqrt{|{\mathcal A}|+1}.|v_{i,j}(n)(\zeta_i(n))|}$, for each $i<k_1$ and $j<l_i$, and
		
		$d)$ $\{ \zeta_i(n):\, k_0 \leq i <k_1\}
		\cap D =\void$.
		\medskip
		
		Notice that $B$ is cofinite in $A$, therefore is in $p$. Let $\gamma_n=\frac{\epsilon^*}{(|\mathcal A|+1)\max\{\|h(n)\|: h \in \mathcal B\}}$ for each $n \in B$, where $\|h(n)\|=\sum_{ \mu \in \supp h(n)}|h(n)(\mu)|$. Now let $(W_h: h \in \mathcal C)$ and $\psi$ be given. Fix $n\in B$.

		For each $h \in \mathcal C$, fix $V_{h} \subseteq W_{h}$ with $V_{h}$ an arc of length $4\epsilon^*$.
		
		Given an arbitrary $\epsilon$-arc function $\psi$ as required, fix $\psi^*$ an $\epsilon^*$-arc function such that $\supp \psi^* \supseteq D$, $\supp \psi^* \cap \{\zeta_i(n):\, 0\leq i<k_1\}=\void$, $\psi^* \leq \psi$ and $\supp h(n) \setminus \{ \zeta_i(n):\, k_0\leq i< k_1\} \subseteq \supp \psi^*$, for each $i<k_1$, $j<l_i$ and $h\in \mathcal B_{i, j}$.

		For each $\mu \in \supp \psi^*$ choose $x_\mu$ such that $K_nx_\mu $ is the center of $\psi^*(\mu)$.

		For each $0\leq i<k_1$, $j<l_i$, notice that $ \{ \zeta_{0}(n), \ldots, \zeta_{k_1-1}(n)\} \cap  \supp h \subset\{ \zeta_{i}(n), \ldots , \zeta_{k_1-1}(n)\}$ for each $h \in \mathbb  B_{i,j}$.
		
		We will define by downward recursion, for $0<i\leq k_1-1$,  an arc $Q_{i,0}\subset\mathbb T$.
		
		For the first step $i_\star=k_1-1$, we will define $Q_{i_\star,j}$ for $j<l_{i_\star}$, also by downward recursion. So let $j_\star=l_{i_\star}-1$ be the first step.
		
		Let $O_{h}=V_{h} - \sum_{\mu \supp h(n) \setminus \{ \zeta_{i_*}(n)\}}h(n)(\mu)x_\mu$, for each $h \in \mathcal B_{i_*,j*}$.

		Fix an arbitrary arc $J$ of length at least $3.\frac{L}{|v_{i_\star,j_\star}(n)(\zeta_{i_\star}(n))|}$ and ${\mathcal U}_{i_\star,j_\star}$ the ball of radius $4.\epsilon^*$ contained in the product of $\prod_{h\in\mathcal B_{i_\star,j_\star}}O_h$. By Lemma \ref{lem.kronecker}, there exists
		an interval $Q_{i_\star,j_\star}$ contained in $J$ of length $\frac{4.\epsilon^*}{\sqrt{|{\mathcal A}|+1}.|v_{i_\star,j_\star}(n)(\zeta_{i_\star}(n))|}$ such that $\{(h(n)(\zeta_{i_\star}(n)).x:h\in\mathcal B_{i_\star,j_\star}):\, x \in Q_{i_\star,j_\star}\}\subseteq {\mathcal U}_{i_\star,j_\star}\subseteq \mathbb T^{\mathcal B_{i_\star, j_\star}}$.
		
		Suppose $j'<l_{i_\star}$ and we have defined $Q_{i_\star,j}$ for all $j'\le j<l_{i_\star}$. If $j'=0$ then we are done for step $i_\star=k_1-1$.
		
		If not: by $c)$, it follows that $\frac{3L}{|u_{i_\star,j'-1}(n)(\zeta_{i_\star}(n))|} \leq\frac{4.\epsilon^*}{\sqrt{|{\mathcal A}|+1}.|v_{i_\star,j'}(n)(\zeta_{i_\star}(n))|}$ and $Q_{i_\star,j'}$ has size exactly the right side of the inequality above. Let ${\mathcal U}_{i_\star,j'-1}$ be a ball of length $4\epsilon^*$ contained in $\prod_{h\in\mathcal B_{i_\star,j'-1}} O_h$. Applying Lemma \ref{lem.kronecker}, there exists an arc $Q_{i_\star,j'-1}$  of length $\frac{4.\epsilon^*}{\sqrt{|{\mathcal A}|+1}.|v_{i_\star,j'-1}(n)(\zeta_{i_\star}(n))|}$ contained in $Q_{i_\star,j'}$  such that $\{(h(n)(\zeta_{i_\star}(n)).x:h\in\mathcal B_{i_\star,j'-1}):\, x \in Q_{i_\star,j'-1}\}\subseteq {\mathcal U}_{i_\star,j'-1}$.
		
		We thus obtain, for $i_\star=k_1-1$, $Q_{i_\star,0}$ an arc of length $\frac{4.\epsilon^*}{\sqrt{|{\mathcal A}|+1}.|v_{i_\star,0,}(n)(\zeta_{i_\star}(n))|}$ such that 
		
		$\{(h(n)(\zeta_{i_\star}).x:h\in\mathcal B_{i_\star,j}):\, x \in Q_{i_\star,j}\}\subseteq \prod_{h\in\mathcal B_{i_\star,j}} O_h$, for each $j<l_{i_\star}$. At the end, we will  have defined $Q_{i_\star,0}$.
		
		Let $x_{\zeta_{i_\star}(n)}$ be the center of  $Q_{i_\star,0}$.  By  definition of $O_h$ we have
		
		$O_h=V_h - \sum_{\mu\in \supp h(n) \setminus \{ \zeta_{i_\star}(n)\}}h(n)(\mu)x_\mu$, for each $j<l_{i_\star}$ and $h\in\mathcal B_{i_\star,j}$.
		
		Then  $\sum_{\mu \in\supp h(n)}h(n)(\mu)x_\mu  \in 
		\sum_{\mu\in \supp h(n) \setminus \{ \zeta_{i_\star}(n)\}}h(n)  (\mu)x_\mu +O_h =V_h$,  for each $j<l_{i_\star}$ and $h\in\mathcal B_{i_\star,j}$.\\
		
		The first step of the recursion has been carried out. Suppose $0\leq i'<k_1$ and $Q_{i,0}$ has been defined for all $i'\le i<k_1$.

		If $i'=0$ then we are done. Otherwise if $i'-1\geq 0$: 
		
		Let $O_h=V_h - \sum_{\mu\in \supp h(n) \setminus \{ \zeta_{i'-1}(n)\}}h(n)(\mu)x_\mu$, for each $j<l_{i'-1}$ and $h\in\mathcal B_{i'-1,j}$.
		
		It is important here that $ \{ \zeta_{0}(n), \ldots, \zeta_{i'-1}(n)\} \cap \ \supp h =\{ \zeta_{i'-1}(n)\}$, for each $j<l_{i'-1}$ and $h\in\mathcal B_{i'-1,j}$.
		
		We are in similar conditions as for $i'$ to obtain $Q_{i'-1,0}$ and $x_{\zeta_{i'-1(n)}}$ the center of   $Q_{i'-1,0}$ an arc of length 
		$\frac{4.\epsilon^*}{\sqrt{|{\mathcal A}|+1}.|v_{i'-1,0}(n)(\zeta_{i'-1}(n))|}$ such that 
		
		$\sum_{\mu\in \supp h(n) }h(n)(\mu)x_\mu  \in 
		\sum_{\mu\in \supp h(n) \setminus \{ \zeta_{i'-1}(n)\}}h(n)  (\mu)x_\mu +O_h =V_h$,  for each $j<l_{i'-1}$ and $h\in\mathcal B_{i'-1,j}$.
		
		This ends the construction of $x_\mu$ for each $\mu \in \supp \psi^* \cup \{ \zeta_0(n) , \ldots, \zeta_{k_1-1}(n)\}$. Choose an arbitrary $x_\mu$ for $ \mu \in D\setminus (\supp \psi^* \cup \{\zeta_0(n), \ldots, \zeta_{k_0-1}(n)\})$.
		
		Let $\phi(\mu)$ be the arc of center $x_\mu$ and length $\gamma_n$. We show that $\phi$ is the solution we are looking for. 
		
		By the choice of $x_\mu$ and since $\psi^*\le\psi$, it follows  $K_n.\phi \leq \psi$.
		
		Secondly, if $h \in \mathcal B_{i,j}$ then	$\sum_{\mu\in \supp h(n) }h(n)(\mu)x_\mu  \in 
		V_{h}$.  It follows that the center of 	$\sum_{\mu\in \supp h(n) }h(n)(\mu)\phi(\mu)$ is contained in $V_h$ and this arc has length at most $\sum_{\mu\in \supp h(n) }|h(n)(\mu)|.\gamma_n=\sum_{\mu\in \supp h(n) }|h(n)(\mu)|.\frac{\epsilon^*}{(|\mathcal A|+1)\max\{\|h(n)\|: h \in \mathcal B\}} <\epsilon^*$. Therefore, a point of the arc $\sum_{\mu\in \supp h(n) }h(n)(\mu)\phi(\mu)$ is at a distance smaller than $2\epsilon^*+\epsilon^*$ from the center of $V_h$. Since $W_h$ and $V_h$ have the same center and $3\epsilon^*< \frac{\epsilon}{2}$ it follows that $\sum_{\mu\in \supp h(n) }h(n)(\mu)\phi(\mu)\subseteq W_h$. Thus, 
		$\phi$ is as required.

	\end{proof}
	
	\section*{Dedication}
	
	The third listed author would like to dedicate this work to his long term collaborator and friend, Dr. Salvador Garcia Ferreira on the occasion of his 60th birthday.

	\bibliography{grupos}{}
	\bibliographystyle{amsplain}
	
\end{document}